\documentclass[10pt]{article}
\usepackage[a4paper]{geometry}
\usepackage{amsthm}
\usepackage{amsmath}
\usepackage{amssymb}
\usepackage{epsfig,graphicx}
\usepackage{subfigure}
\usepackage{wrapfig}
\usepackage{mathdots}

\begin{document}

\theoremstyle{plain}
\newtheorem{Theorem}{Theorem}
\newtheorem{Lemma}[Theorem]{Lemma}
\newtheorem{Proposition}[Theorem]{Proposition}
\newtheorem{Corollary}[Theorem]{Corollary}
\newtheorem{Condition}[Theorem]{Condition}
\newtheorem{Algorithm}[Theorem]{Algorithm}

\theoremstyle{definition}
\newtheorem{Definition}[Theorem]{Definition}

\theoremstyle{remark}
\newtheorem*{Remark}{Remark}
\newtheorem{Example}[Theorem]{Example}

\newcommand{\R}{\mathbb{R}}

\newcommand{\todo}[1]{\vspace{5 mm}\par \noindent
\marginpar{\textsc{ToDo}}
\framebox{\begin{minipage}[c]{0.95 \textwidth}
\tt #1 \end{minipage}}\vspace{5 mm}\par}

\title{Temporal homogenization of linear ODEs, with applications
to parametric super-resonance and energy harvest}

\author{Molei Tao and Houman Owhadi}

\maketitle

\begin{abstract}
\noindent
We consider the temporal homogenization of linear ODEs of the form $\dot{x}=Ax+\epsilon P(t)x+f(t)$, where $P(t)$ is periodic and $\epsilon$ is small. 
Using a 2-scale expansion approach, we obtain the long-time approximation $x(t)\approx \exp(At) \left( \Omega(t)+\int_0^t \exp(-A \tau) f(\tau) \, d\tau \right)$, where $\Omega$ solves the cell problem $\dot{\Omega}=\epsilon B \Omega + \epsilon F(t)$ with an effective matrix $B$ and an explicitly-known $F(t)$. We provide necessary and sufficient condition for the accuracy of the  approximation (over a $\mathcal{O}(\epsilon^{-1})$ time-scale), and show how $B$ can be computed (at a cost independent of $\epsilon$). As a direct application, we investigate the possibility of using RLC circuits to harvest the energy contained in small scale oscillations of  ambient electromagnetic fields (such as Schumann resonances). Although a RLC circuit parametrically coupled to the field may achieve such energy extraction via parametric resonance, its resistance $R$ needs to be smaller than a threshold $\kappa$ proportional to the fluctuations of the field, thereby limiting practical applications. We show that if $n$ RLC circuits are appropriately coupled via mutual capacitances or inductances, then energy extraction can be achieved when the resistance of each circuit is smaller than $n\kappa$. Hence, if the resistance of each circuit has a non-zero fixed value, energy extraction can be made possible through the coupling of a sufficiently large number $n$ of circuits ($n\approx 1000$ for the first mode of Schumann resonances and contemporary values of capacitances, inductances and resistances).
The theory is also applied to the
 control of the oscillation amplitude of a (damped) oscillator.
\end{abstract}

\section{Introduction}
\subsection{Main mathematical results}
Consider time-dependent non-homogeneous linear ODE
\begin{equation}
    \dot{x}=Ax+\epsilon P(t)x+f(t)
    \label{eq_tempHomoSystem}
\end{equation}
on $\R^n$, where $A$ is a constant $n\times n$ real matrix, $P(t)$ is a square-integrable $2\pi/\omega$-periodic function taking real matrix values, $f(t)$ is a vector-valued function satisfying that $\exp(-At)f(t)$ is integrable on $[0,\hat{C}\epsilon^{-1}]$ for some $\hat{C}>0$, and $0<\epsilon \ll 1$.

Our main purpose is to approximate the solution of \eqref{eq_tempHomoSystem} over a $\mathcal{O}(\epsilon^{-1})$ timescale, without resolving oscillations of $P(t)$ over that (long) interval of time. Our first result is as follows:

\begin{Theorem}
    \label{thm_temporalHomogenization}
    Let $x(t)$ be the solution of the non-autonomous ODE system \eqref{eq_tempHomoSystem}. If $\exp(-A t) P(t) \exp(A t)$ is uniformly bounded in $t$, then there exists a constant matrix $B$, independent of $f(\cdot)$, such that
    \begin{equation}
        x(t) = \exp(At) \left( \Omega(t)+\int_0^t \exp(-A \tau) f(\tau) \, d\tau + E(t,\epsilon) \right), \label{eq_tempHomosedResultbis}
    \end{equation}
    with
    \begin{align}
        \dot{\Omega} &=\epsilon B \Omega + \epsilon F(t) \nonumber\\
        F(t) &:= \exp(-At)P(t)\exp(At)\int_0^t \exp(-A\tau)f(\tau)\,d\tau ,
        \label{eq_MainCellProblem}
    \end{align}
    where $\Omega(0)=x(0)$ and, noting $\|y\|:=\sqrt{y_1^2+\cdots+ y_n^2}$ the Euclidean 2-norm of $y$,  the error ($E(t,\epsilon)$ in \eqref{eq_tempHomosedResultbis}) satisfies, for $0\leq t\leq C \epsilon^{-1}$,
    \begin{equation}
         \| E(t, \epsilon) \| \leq C \epsilon \exp(\epsilon^2 C t) \left( \max_{\tau\in [0,t]} \|\Omega(\tau)\| + \max_{\tau\in [0,t]} \left\| \int_0^\tau \exp(-A s)f(s)\,ds \right\| \right),
         \label{eq_errorBound}
    \end{equation}
    for some constant $C$ independent of $t$ and $\epsilon$. Moreover, $B$ can be identified by either
    \begin{equation}
        B=\mathcal{G}[\exp(-At)P(t)\exp(At)],
        \label{eq_algebraic_representation_intro}
    \end{equation}
    where $\mathcal{G}$ is defined in Definition \ref{def_nonOsc1}, or
    \begin{equation}
        B=\lim_{T \rightarrow \infty} \frac{1}{T}\int_0^T \exp(-A\tau)P(\tau)\exp(A\tau)\,d\tau ,
        \label{eq:oedhohiud}
    \end{equation}
    where the limit exists if and only if $e^{-At}P(t)e^{At}$ is uniformly bounded in $t$.
\end{Theorem}

Theorem \ref{thm_temporalHomogenization} shows that if $\exp(-A t) P(t) \exp(A t)$ remains uniformly bounded, then up to time $\mathcal{O}(\epsilon^{-1})$, the solution of \eqref{eq_tempHomoSystem} can be approximated by
\begin{equation}
    x(t) \approx \exp(At) \left( \exp(\epsilon B t)x(0) + \int_0^t \exp(\epsilon B (t-\tau)) \epsilon F(\tau) \,d\tau + \int_0^t \exp(-A \tau) f(\tau) \, d\tau \right).
    \label{eq_tempHomoResult}
\end{equation}

The analytical expression in the right side of \eqref{eq_tempHomoResult} can be explicitly computed for a large class of $f$'s (e.g., $f(t)=p(t,\cos t,\sin t)$ for polynomial $p$).  $B$ acts as an effective matrix characterizing the time-homogenized action of fast periodic oscillations. We provide two methods for the identification of $B$: the first one \eqref{eq_algebraic_representation_intro} is  algebraic and described in Proposition \ref{prop:alebraicB}; the second one \eqref{eq:oedhohiud} is computational and described in Proposition \ref{thm_averaging}.
\medskip

Uniform boundedness of $\exp(-A t) P(t) \exp(A t)$ is not only sufficient for the accuracy of the approximation, but also necessary as shown by the following theorem.

\begin{Theorem}
    Consider system \eqref{eq_tempHomoSystem}. Given a constant matrix $B$, define the approximation error
    \[
        E(t,\epsilon):=\exp(-At)x(t)-\Omega(t)-\int_0^t \exp(-A\tau)f(\tau)\,d\tau,
    \]
    where $\Omega$ satisfies \eqref{eq_MainCellProblem}. If $\exp(-A t) P(t) \exp(A t)$ is not uniformly bounded in time, then for any constant matrix $B$ independent of $f(\cdot)$, there exists at least one initial condition $x_0$ and a constant $\bar{C}$ (independent of $\epsilon$), such that there is no constant $C$ (independent of $\epsilon$) that satisfies
    \[
        \|E(t,\epsilon)\| \leq C \epsilon \left( \max_{\tau\in [0,t]} \|\Omega(\tau)\| + \max_{\tau\in [0,t]} \left\| \int_0^\tau \exp(-A s)f(s)\,ds \right\| \right)
    \]
    for $t\leq \bar{C} \epsilon^{-1}$.
    \label{thm_tempHomoCannot}
\end{Theorem}

Section \ref{Section_theory} establishes these results. Sections \ref{Section_PR} and \ref{sec:app2} describe how the method can be applied to
 (i) control the oscillation amplitude of a (damped) oscillator, and (ii) couple oscillators in order to lower the threshold on fluctuation amplitude needed for harvesting energy.

\subsection{Mathieu's equation}
Mathieu's equation is an example that can be expressed as \eqref{eq_tempHomoSystem}, with
\[
    A=\begin{bmatrix} 0 & 1 \\ -\omega^2 & 0 \end{bmatrix}, \qquad
    P(t)=\begin{bmatrix} 0 & 0 \\ -\omega^2 \cos(2 \omega t) & 0 \end{bmatrix}, \qquad
    f(t)=\begin{bmatrix} 0 \\ 0 \end{bmatrix}.
\]
It is a prototype for the study of parametric resonance (see Section \ref{Section_PR_sln}). \cite{verhulst2009perturbation}, for instance, used averaging and perturbation analysis to capture $\mathcal{O}(\epsilon^{-1})$-time dynamics of the system, and the technique was extended to multi-dimensional oscillators in \cite{DoVe08} (see also \cite{DeLe08}) and applied in structural engineering for stablization purposes \cite{dohnal2009optimal}. We also refer to \cite{PaSt53, BlBl94, BeSe03, AkDi99, MeZh91, YaSh03, NgLa00} for examples of  applications of parametric resonance in science and engineering. Parametric resonance can lead to not only exponential growths of oscillation amplitudes (a well known phenomenon used by children to make a playground swing go higher by pumping their legs) but also exponential decays (see Corollary \ref{cor_ExponentialDecay} and its remarks; this aspect appears to have received less attention in the literature).

\subsection{Relation with Floquet theory and perturbation analysis}
It is in general difficult to obtain a closed-form solution of a non-autonomous system of the form
\begin{equation}\label{eq:floquetequ}
\dot{x}(t)=F(t)x(t),
\end{equation}
 where $F(t)$ is a periodic matrix-valued function.

Floquet theory \cite{FloquetTheory} (known as Bloch's theorem \cite{AsMe76} in physics)  shows that the fundamental matrix associated with \eqref{eq:floquetequ}, i.e., the matrix-valued  solution of $\dot{\Phi}=F(t)\Phi$ with $\Phi(0)=I$, satisfies
\begin{equation}\label{eq:matBfloq}
    \Phi(t)=Q(t)\exp(t R),
\end{equation}
where $Q(t)$ is a periodic matrix and $R$ is a constant matrix. Although Floquet theory provides important information on the solution structure, it does not, in general, help identify $R$ or $Q(t)$.

If $f\equiv 0$ in our system of interest \eqref{eq_tempHomoSystem}, then $F(t)$ is the sum of a constant matrix and a small periodic perturbation, and  perturbation analysis \cite{Na73,Vela03,MR2382139,SaVeMu10} can be combined with Floquet theory to obtain a long-time approximation of the  fundamental matrix. More precisely, using an asymptotic expansion Ansatz  $\Phi(t)=\Phi_0(t)+\epsilon \Phi_1(t)+\mathcal{O}(\epsilon^2)$ and matching orders yields
\begin{equation}
    \Phi_0(t)=\exp(A t), \qquad \Phi_1(t)=\exp(A t)\int_0^t \exp(-A \tau)P(\tau) \exp(A\tau) d\tau.
    \label{eq_perturbationResult}
\end{equation}
At the same time, \eqref{eq:matBfloq} leads to
\[
    \Phi(nT+t)=Q(nT+t)\exp\left((nT+t)R\right)=Q(t)\exp(tR) \exp(nTR) = \Phi(t)\Phi(T)^n,
\]
where $n$ is an integer, and $T$ is the period. Let
\[
    \tilde{\Phi}(t) = (\Phi_0(t-nT)+\epsilon \Phi_1(t-nT)) (\Phi_0(T)+\epsilon \Phi_1(T))^n, \qquad \text{when } nT \leq t < (n+1)T.
\]
When $t=\mathcal{O}(\epsilon^{-1})$, $n=\lfloor t/T \rfloor=\mathcal{O}(\epsilon^{-1})$, and a standard local-to-global error analysis leads to
\[
    \Phi(t)=\tilde{\Phi}(t) + \mathcal{O}(\epsilon).
\]
Therefore, when $f\equiv 0$, Floquet theory provides an alternative to Theorem \ref{thm_temporalHomogenization}.

Now consider the $f\neq 0$ case. It is natural to consider the approximation:
\begin{equation}
    x(t) \approx \tilde{x}(t) := \tilde{\Phi}(t) \left( x(0) + \int_0^t \tilde{\Phi}(\tau)^{-1} f(\tau) \, d\tau \right) .
    \label{eq_vp1u3bv018ub0g81b50gf1pugo}
\end{equation}

However, there are two issues with this approach:

\paragraph{(i)}
The calculation of \eqref{eq_vp1u3bv018ub0g81b50gf1pugo} can get quite complex. Indeed, since $\tilde{\Phi}$ is piecewise defined, the non-homogeneous term in \eqref{eq_vp1u3bv018ub0g81b50gf1pugo} can be expressed as
\begin{align*}
    &\int_0^t \tilde{\Phi}(\tau)^{-1} f(\tau) \, d\tau = \sum_{n=0}^{\lfloor t/T \rfloor-1} \int_{nT}^{(n+1)T} \tilde{\Phi}(\tau)^{-1} f(\tau) \, d\tau + \int_{\lfloor t/T \rfloor T}^t \tilde{\Phi}(\tau)^{-1} f(\tau) \, d\tau \\
    &= \sum_{n=0}^{\lfloor t/T \rfloor-1} (\Phi_0(T)+\epsilon \Phi_1(T))^{-n} \left( \int_0^T (\Phi_0(\tau)+\epsilon \Phi_1(\tau))^{-1} f(\tau+nT) \, d\tau \right) \\
    &\qquad + (\Phi_0(T)+\epsilon \Phi_1(T))^{-\lfloor t/T \rfloor} \left( \int_0^{t\text{ mod }T} (\Phi_0(\tau)+\epsilon \Phi_1(\tau))^{-1} f(\tau+\lfloor t/T \rfloor T) \, d\tau \right),
\end{align*}
which cannot be reduced further when $f$ is arbitrary.
%

\paragraph{(ii)}
The $\mathcal{O}(\epsilon)$ error in $\tilde{\Phi}(t)$ may (depending on the choice of $f$) result in an $\mathcal{O}(1)$ error after integration to $t=\mathcal{O}(\epsilon^{-1})$ in \eqref{eq_vp1u3bv018ub0g81b50gf1pugo}. This issue could be addressed by further including a 2nd order term $\epsilon^2 \Phi_2(t)$ in the approximation $\tilde{\Phi}(t)$, but this comes with the price of more complex  calculations.

Note our method approximates the fundamental matrix by, up to $t=\mathcal{O}(\epsilon^{-1})$,
    \begin{equation}
        \Phi(t) = \exp(A t) \left( \exp\left( \epsilon \int_0^t \mathcal{G}[\exp(-A \tau)P(\tau) \exp(A\tau)] d\tau \right) + o(\epsilon) \right) ,
        \label{eq_adfjhoqiweubgtoqgforbo}
    \end{equation}
    whereas the aforementioned perturbative Floquet approach uses (when $t<T$),
    \begin{equation}
        \exp(A t) \left( I + \epsilon \int_0^t \exp(-A \tau)P(\tau) \exp(A\tau) d\tau + o(\epsilon) \right).
        \label{eq_adfjhoqiweubgtoqgforbo1}
    \end{equation}
Since the integral of $\mathcal{G}[\exp(-A \tau)P(\tau) \exp(A\tau)]-\exp(-A \tau)P(\tau) \exp(A\tau)$ is small (Lemma \ref{lem_remainderOperator}),
\eqref{eq_adfjhoqiweubgtoqgforbo1} could  be seen as a 1st-order approximation of \eqref{eq_adfjhoqiweubgtoqgforbo}. Including 2nd-order terms in     \eqref{eq_adfjhoqiweubgtoqgforbo1} would improve its accuracy at a price of increased computational complexity, whereas
\eqref{eq_adfjhoqiweubgtoqgforbo} provides a simple high order approximation.

\subsection{Relation with averaging} \label{Section_averaging}
Averaging methods (e.g., \cite{Na73,Vela03,MR2382139,SaVeMu10}) approximate the solution of
\begin{equation}\label{eq:uggwugwy33}
\dot{y}=\epsilon f(y,t)
\end{equation}
 by the solution of $\dot{z}=\epsilon \bar{f}(z)$. These methods can be divided into two categories:
(i) when $f(y,t)$ is $T$-periodic in $t$, the effective dynamics can be obtained using
\[
    \bar{f}(x):=\frac{1}{T}\int_0^T f(x,t) \, dt,
\]
with a $z(t)-y(t)=\mathcal{O}(\epsilon)$ upper-error-bound for $t=\mathcal{O}(\epsilon^{-1})$; (ii) when $f(y,t)$ is not periodic, the effective dynamics can be obtained using
\begin{equation}
    \bar{f}(x):=\lim_{T\rightarrow \infty} \frac{1}{T}\int_0^T f(x,t) \, dt,
    \label{eq_existenceLimit}
\end{equation}
with a $z(t)-y(t)=o(1)$ upper-error-bound for $t=\mathcal{O}(\epsilon^{-1})$ under certain additional assumptions (see Definition 4.2.4 and Theorem 4.3.6 of \cite{SaVeMu10}).

When $f=0$, our approximation can be reproduced by averaging: introduce a change of variables $\Xi(t)=\exp(-At)x(t)$ (when $A$ has only imaginary eigenvalues, this is a common trick used in perturbation analysis \cite{Verhulst:96}); then system \eqref{eq_tempHomoSystem} transforms into
\begin{equation}
    \dot{\Xi}=\epsilon \exp(-At)P(t)\exp(At) \Xi,
    \label{eq_8uboqurbog1ubgo134ubt0o}
\end{equation}
Since $\exp(-At)P(t)\exp(At)$ may be non-periodic in $t$, general averaging theory is required, and it approximates \eqref{eq_8uboqurbog1ubgo134ubt0o} by (when the limit exists)
\begin{equation}
    \dot{\Upsilon}=\epsilon \left( \lim_{T\rightarrow\infty} \frac{1}{T} \int_0^T \exp(-At)P(t)\exp(At) \, dt \right) \Upsilon.
    \label{eq_v89agogyubot241obhyioybdf}
\end{equation}
This limit is identical to \eqref{eq:oedhohiud}, and can be shown to be equivalent to our algebraic approach \eqref{eq_algebraic_representation_intro} (see Proposition \ref{prop:alebraicB} and Section \ref{section_temporalHomogenization}).

Therefore, in the homogeneous case, the contribution of this paper is not to provide a new approximation but to (i) prove a sharper $\mathcal{O}(\epsilon)$ error bound, (ii) prove that the assumption that $\exp(-At)P(t)\exp(At)$ remains uniformly bounded in time is both necessary and sufficient for the accuracy of the approximation \eqref{eq_v89agogyubot241obhyioybdf}, and (iii) illustrate an algebraic alternative for computing the effect matrix (see Propositions \ref{thm_algebraic} and \ref{prop:alebraicB}), which could be used as a guiding tool for designing systems with distinct effective dynamics (see Sections \ref{Section_PR} and \ref{sec:app2}).

When $f \neq 0$,  approximation \eqref{eq_tempHomosedResultbis} is new. One can still introduce slow variables $\Xi(t)=\exp(-At)x(t)-\int_0^t \exp(-A\tau)f(\tau)\,d\tau$ and show
\[
    \dot{\Xi}=\epsilon \exp(-At)P(t)\exp(At) \left( \Xi + \int_0^t \exp(-A\tau) f(\tau) \, d\tau \right).
\]
However, $\epsilon \int_0^t \exp(-A\tau) f(\tau) \, d\tau$ might be exponentially large and this prohibits the application  of classical averaging. For example, if $A=-1$ and $f(t)=1$ (both scalars), $\dot{\Xi}=\epsilon P(t)\Xi+\mathcal{O}(\epsilon \exp(\epsilon^{-1}))$ when $t=\mathcal{O}(\epsilon^{-1})$.

\subsection{Relation with classical homogenization}\label{Section_spatialHomogenization}
As in classical homogenization theory (e.g., \cite{BeLiPa78,  Nguetseng:1989, JiKoOl91, Allaire1992}), the constant
matrix $B$ in Theorem \ref{thm_temporalHomogenization} can be seen as an effective matrix capturing the homogenized effect of the periodic perturbation on the dynamics.

Our results are built on a two-scale expansion technique analogous to the one used in classical homogenization theory (see also \cite{KeCo96}). One major difference is the lack of ellipticity in \eqref{eq_tempHomoSystem}.
See also \cite{garnier1997homogenization, cooper2000parametric, DoLeLe10} for homogenization problems involving time (with different systems of interest).

In the special case of $f=0$, another analogy with classical homogenization is as follows:
let $F(t)=\epsilon^{-1} A+P(t/\epsilon)$, then after rescaling time our system becomes
\begin{equation}
    \dot{X}=F(t)X .
    \label{eq_primalDynamics}
\end{equation}
Let $\mathcal{A}(t)$ be the matrix-valued solution of
\begin{equation}
    \dot{\mathcal{A}}(t)=-\mathcal{A}(t) F(t),
    \label{eq_adjointDynamics}
\end{equation}
and $Y$ be the solution of the 1D problem
\begin{equation}
    \frac{d}{dt} \left( \mathcal{A}(t) \frac{d Y}{dt} \right)=0,
    \label{eq_ellipticHomogenization}
\end{equation}
then it can be shown that $X=\dot{Y}$. Here
\eqref{eq_ellipticHomogenization} is akin to the divergence form  PDE used as a prototypical example in classical homogenization theory  \cite{BeLiPa78, Nguetseng:1989}. Unfortunately, obtaining $\mathcal{A}(t)$ via \eqref{eq_adjointDynamics} is  as difficult as solving the original problem \eqref{eq_primalDynamics}.


Note also that, in the context of stochastic homogenization \cite{MR542557, PapVar82}, as in \eqref{eq:oedhohiud}, the calculation of the effective conductivity requires taking the asymptotic limit of local cell problems.


\subsection{Other related work}
Magnus expansion \cite{Ma54} allows for a  representation of the solution of \eqref{eq:floquetequ} (note $f$ has to be 0) as an infinite series of integrals of increasingly many matrix commutators. For practical applications (see \cite{BlCa09} for a review), the infinite series has to be truncated to a finite number of terms.
In many cases convergence after truncation is not guaranteed or slow (e.g., \cite{BlCa98}), and one often faces such problem when studying $\mathcal{O}(\epsilon^{-1})$ long time behavior of our system of interest \eqref{eq_tempHomoSystem}.

Alternative strategies become available when additional restrictions are placed on the system \eqref{eq:floquetequ} or only coarse estimates are needed.
For instance, stability theory exists for Lappo-Danilevskii systems (which is a small subclass of \eqref{eq:floquetequ}, characterized by the commutation of $F(t)$ with its integrals \cite{Ad95}), or when $F(t)$ is almost constant and the constant part is asymptotically stable \cite{Ad95}. There are also loose bounds of the characteristic matrix $R$ in \eqref{eq:matBfloq} (e.g., \cite{YaSt75a, YaSt75b} and IV.6 of \cite{Ad95}).
There is also a rich literature on the resolution and analysis of periodic time-dependent Schr\"{o}dinger equation (e.g., \cite{Sh65, PeMo93, RaGiFi03}) and, in particular,  on the steady state Schr\"{o}dinger operator with multi-dimensional periodic potentials (e.g., \cite{Sk85, De95, Ca00}). \cite{Wi66, ChTe04} are examples of reviews.
We also refer to \cite{ShamsulAlam2003987, grozdanov1988quantum, maricq1982application} for an incomplete list of additional methods.

This article is restricted to linear systems. Only partial results are available for nonlinear systems. For instance, \cite{Vela03, li2002extension} provide nonlinear generalizations of Floquet theory. Nonlinear analogies to parametric resonance (e.g., autoparametric resonance) have been studied using averaging and perturbation analysis \cite{verhulst2002parametric, fatimah2002bifurcations, verhulst2005autoparametric}; see also \cite{mahmoud2000periodic, zounes2002subharmonic, zhang2002effect, abraham2003approximate} for more references. We also refer to \cite{KoOw13} for the control of a nonlinear model of double-strand DNA via parametric resonance.

\section{Theory}
\label{Section_theory}
\subsection{Algebraic structure}
\begin{Condition}
    Let $t\in\mathbb{R}$, and $P(t)=P(t+2\pi/\omega)$ be a real-matrix-valued periodic function in $L^2$. Assume that $A\in \mathbb{R}^{n\times n}$ is a real matrix (not necessarily diagonalizable and,  possibly, with complex eigenvalues). Assume without loss of generality that $A$ is in Jordan canonical form.
    \label{cd_setting}
\end{Condition}

\begin{Remark}
    The assumption of Jordan canonical form is without loss of generality, because it can be achieved via a change of basis, which affects $P(t)$ but not its periodicity. This assumption is not required by Theorem \ref{thm_temporalHomogenization} or \ref{thm_tempHomoCannot} either, and is only needed by the specific algebraic calculation in Propositions \ref{thm_algebraic} and \ref{prop:alebraicB}.
\end{Remark}

\begin{Lemma}
    Under Condition \ref{cd_setting}, $\exp(-A t) P(t) \exp(A t)$ can be uniquely expressed (in $L^2$ sense, which will no longer be stated in the rest of the paper unless confusion arises) as a linear combination (with coefficients being constant real matrices) of $t^k e^{at} \cos(b t)$ and $t^k e^{at} \sin(b t)$, where $(a,b,k)\in \Sigma$ for some countable set $\Sigma$, in which $a$, $b$ and $k$ components respectively take values in a finite subset of $\R$, a countable subset of $\R$, and $\{0,1,\ldots,2n-3,2n-2\}$.
\label{lem_fourierExpansion}
\end{Lemma}

\begin{proof}
    As a well-known corollary of Jordan canonical form theory (see for instance \cite{Perko}), both $\exp(-A t)$ and $\exp(A t)$ can be uniquely expressed as linear combinations of $t^{r} e^{\pm \lambda t} \cos(\mu t)$ and $t^{r} e^{\pm \lambda t} \sin(\mu t)$, where for each triplet $(r,\lambda,\mu)$, $\lambda$ and $\mu$ correspond to the real and imaginary parts of one of $A$ eigenvalues, and $r$ is less or equal to the number of off-diagonal 1's in the associated Jordan block.

    Also, represent $P(t)$ in Fourier series. Since products of $\cos$ and $\sin$ can be uniquely represented as sums of $\cos$ and $\sin$, the lemma is proved. $k$, $a$ and $b$ depend on $\lambda$, $\mu$, $r$, and Fourier coefficients of $P(t)$.
\end{proof}

\begin{Definition}[Growth operator]
Using the representation given by Lemma \ref{lem_fourierExpansion}:
    \begin{equation}
      \exp(-A t) P(t) \exp(A t) = \sum_{(a,b,k)\in\Sigma} \left( C_{abk} t^k e^{at} \cos(b t) + D_{abk}  t^k e^{at} \sin(b t) \right) ,
      \label{eq_Gexpansion}
    \end{equation}
	we define the growth component of $\exp(-A t) P(t) \exp(A t)$ by
    \begin{align}
        & \mathcal{G}[\exp(-A t) P(t) \exp(A t)] := \nonumber\\
        & \quad \quad \left( \sum_{(a,b,k)\in\Sigma,a>0}+\sum_{(a,b,k)\in\Sigma,a=0,k\neq 0}+\sum_{(a,b,k)\in\Sigma,a=0,k=0,b=0} \right)  \left( C_{abk} t^k e^{at} \cos(b t) + D_{abk}  t^k e^{at} \sin(b t) \right) .
        \label{eq_growthComponent}
    \end{align}
    \label{def_nonOsc1}
\end{Definition}

\begin{Proposition}
    $\exp(-A t) P(t) \exp(A t)$ remains bounded for all $t$, if and only if
    $\mathcal{G}[\exp(-A t) P(t) \exp(A t)]$ is time-independent, i.e., when described in the form given by \eqref{eq_growthComponent}, it does not contain $(a>0,k,b)$ or $(a=0,k\neq 0, b)$ terms.
    \label{thm_boundedness}
\end{Proposition}
\begin{proof}
    This directly follows from Definition \ref{def_nonOsc1}.
\end{proof}

\begin{Remark}
    When $A$ is diagonalizable and real parts of all its eigenvalues are the same, $\exp(-A t) P(t) \exp(A t)$ remains bounded for all $t$. In general, however, whether it is bounded depends not only on $A$ but also on entries of $P(t)$.
\end{Remark}

\begin{Proposition}[Growth operator is equivalent to time-averaging]
    \begin{equation}
        \lim_{T\rightarrow \infty} \frac{1}{T} \int_0^T \exp(-A \tau) P(\tau) \exp(A \tau) \, d\tau
        \label{eq_uidoiufb1ohlbfqlf}
    \end{equation}
    exists if and only if $\exp(-A t) P(t) \exp(A t)$ remains bounded for all $t$, and in this case
    \begin{equation}
        B=\mathcal{G}[\exp(-A t) P(t) \exp(A t)]=\lim_{T\rightarrow \infty} \frac{1}{T} \int_0^T \exp(-A \tau) P(\tau) \exp(A \tau) \, d\tau .
        \label{eq_thm_averaging}
    \end{equation}
    \label{thm_averaging}
\end{Proposition}
\begin{proof}
    If bounded, $\exp(-A t) P(t) \exp(A t)$ can be written as
    \begin{equation}
        C_0+ \sum_{i} \left( C_i \cos(\omega_i t)+D_i \sin(\omega_i t) \right) + \sum_{(a,b,k)\in\Sigma,a<0} \left( C_{abk} t^k e^{at} \cos(b t) + D_{abk}  t^k e^{at} \sin(b t) \right),
    \end{equation}
    where $C_0$ and $C_i$'s are constant matrices, $\omega_i$'s are constant quasi-periods that not necessarily have a finite least common multiple, and $i$ may take finitely-many or countably-many values (depending on whether Fourier series of $P$ terminates at finite terms). In this case,
    \begin{equation}
        \mathcal{G}[\exp(-A t) P(t) \exp(A t)] = C_0 .
    \end{equation}
    Since $0=\lim_{T\rightarrow \infty} \int_{0}^{T} e^{at} t^k \cos(\omega t) \, dt /T$ and $0=\lim_{T\rightarrow \infty} \int_{0}^{T} e^{at} t^k \sin(\omega t) \, dt /T$ for $a<0$ or $(a=0,t=0)$, we have
    \begin{align}
        C_0&=\lim_{T\rightarrow \infty} \frac{1}{T} \int_0^T \Big( C_0+ \sum_{i} \left( C_i \cos(\omega_i t)+D_i \sin(\omega_i t) \right) \nonumber\\
        &\qquad\qquad + \sum_{(a,b,k)\in\Sigma,a<0} \left( C_{abk} t^k e^{at} \cos(b t) + D_{abk}  t^k e^{at} \sin(b t) \right) \Big) \, dt ,
    \end{align}
    where in the case of infinite summation swapping the limit and infinite sum is justified by dominated convergence.

    If $\exp(-A t) P(t) \exp(A t)$ is unbounded, its representation obtained from  Lemma \ref{lem_fourierExpansion}
    contains terms that grow as $t^k$ ($k>0$) or $\exp(a t)$ ($a>0$), and therefore the integral in \eqref{eq_uidoiufb1ohlbfqlf} does not exist.
\end{proof}

\begin{Proposition}[Algebraic calculation of growth operator]
    \label{thm_algebraic}
    Let $M(t)=\exp(-At)P(t)\exp(At)$. Denote $A$'s eigenvalues by $\lambda_i \pm \sqrt{-1} \mu_i$ (assuming $\mu_i\geq 0$). Let $P^{\cos}_{ij,l}$ and $P^{\sin}_{ij,l}$ be the $l^{th}$ Fourier coefficients of $P(t)$. Let $L_{ij}$ be the set of all nonnegative integers $l$ such that $l \omega=\left| \mu_i \pm \mu_j \right|$ (recall $\omega \geq 0$ is the largest frequency of $P(t)$). For all $i$, let $\alpha_i$ be the identity matrix of the size of $A_{ii}$, and
    \[
        \beta_i:=\begin{bmatrix} 0 & 1 & & \\ -1 & 0 & \ddots & \\ & \ddots & 0 & 1 \\ & & -1 & 0 \end{bmatrix}
    \]
    be the canonical symplectic matrix when $\mu_i\neq 0$, and 0 if $\mu_i=0$, also of the size of $A_{ii}$.

    Then, $\mathcal{G}[M(t)]_{ij}=\mathcal{G}[M(t)_{ij}]$ for all $(i,j)$ pairs. Moreover, under Condition \ref{cd_setting} and boundedness of $M(t)$, $L_{ij}$ is of finite size, and $\mathcal{G}[M(t)_{ij}] \equiv \sum_{l\in L_{ij}} \bar{M}_{ij,l}$, where $\bar{M}_{ij,l}:=0$ if $\lambda_i \neq \lambda_j$; if $\lambda_i = \lambda_j$,
    \begin{equation}
        \bar{M}_{ij,l} := \begin{cases}
         (\beta_i P^{\cos}_{ij,l} \alpha_j + \alpha_i P^{\sin}_{ij,l} \alpha_j + \alpha_i P^{\cos}_{ij,l} \beta_j - \beta_i P^{\sin}_{ij,l} \beta_j)/4 ,
           & \qquad \text{if } \mu_i-\mu_j=l \omega, \\
         (\beta_i P^{\cos}_{ij,l} \alpha_j - \alpha_i P^{\sin}_{ij,l} \alpha_j + \alpha_i P^{\cos}_{ij,l} \beta_j - \beta_i P^{\sin}_{ij,l} \beta_j)/4 ,
           & \qquad \text{if } \mu_j-\mu_i=l \omega, \\
        (-\beta_i P^{\cos}_{ij,l} \alpha_j - \alpha_i P^{\sin}_{ij,l} \alpha_j + \alpha_i P^{\cos}_{ij,l} \beta_j - \beta_i P^{\sin}_{ij,l} \beta_j)/4 ,
           & \qquad \text{if } \mu_i+\mu_j=l \omega, \\
        0, & \qquad \text{otherwise.} \\
        \end{cases}
        \label{eq_algebraic_representation}
    \end{equation}
\end{Proposition}

\begin{proof}
    It is not difficult to see from its definition that $\mathcal{G}$ is a linear operator and $\mathcal{G}[M]_{ij}=\mathcal{G}[M_{ij}]$.

    Since $M(t)$ is bounded, each term in $M_{ij}$ that possibly persists after the application of $\mathcal{G}$ is a product of at most 3 trigonometric functions (decaying components will be removed). Let their frequencies be respectively $\mu_i$, $l \omega$, and $\mu_j$. This product yields a non-zero constant term if and only if $\pm \mu_i \pm l\omega \pm \mu_j = 0$. Since only the constant terms will persist after the application of $\mathcal{G}$, it is sufficient to consider only $l^{th}$-modes in the Fourier expansion of $P(t)$ with $l\in L_{ij}$, i.e.,
    \[
        \mathcal{G}[M_{ij}]=\sum_{l\in L_{i,j}} \bar{M}_{ij,l},
    \]
    where
    \[
        \bar{M}_{ij,l}=\mathcal{G}\left[ \exp(-A_{ii} t) \left( P^{\cos}_{ij,l} \cos(l\omega t)+ P^{\sin}_{ij,l} \sin(l\omega t) \right) \exp(A_{jj} t) \right].
    \]
    When $\lambda_i > \lambda_j$, by the definition of $\mathcal{G}$, $\bar{M}_{ij,l}=0$. When $\lambda_i < \lambda_j$, boundedness of $M$ ensures $\left( P^{\cos}_{ij,l} \cos(l\omega t)+ P^{\sin}_{ij,l} \sin(l\omega t) \right)=0$, and therefore $\bar{M}_{ij,l}=0$ too.

    Now consider only the case of $\lambda_i = \lambda_j$. Since boundedness of $M$ rules out presence of $t^k$,
    \[
        \bar{M}_{ij,l} = \mathcal{G}\left[ \exp(-\tilde{A}_{ii} t) \left( P^{\cos}_{ij,l} \cos(l\omega t)+ P^{\sin}_{ij,l} \sin(l\omega t) \right) \exp(\tilde{A}_{jj} t) \right] ,
    \]
    where $\tilde{A}_{ii}$'s are matrices in canonical Jordan form with eigenvalues $\pm \sqrt{-1} \mu_i$ without $A_{ii}$'s off-diagonal blocks, i.e.,
    \[
        \tilde{A}_{ii} = \begin{bmatrix} 0 & \mu_i & & \\ -\mu_i & 0 & \ddots & \\ & \ddots & 0 & \mu_i \\ & & -\mu_i & 0 \end{bmatrix} .
    \]
    Therefore,
    \[
        \bar{M}_{ij,l}=\mathcal{G}\left[ \left( \alpha_i \cos(\mu_i t)-\beta_i \sin(\mu_i t) \right) \left( P^{\cos}_{ij,l} \cos(l\omega t)+ P^{\sin}_{ij,l} \sin(l\omega t) \right) \left( \alpha_j \cos(\mu_j t)+\beta_j \sin(\mu_j t) \right) \right] .
    \]

    It can be computed by basic trigonometric identities that (for arbitrary parameters $a,b,c,d,e,f,\mu,\nu,\Omega$)
    \begin{align*}
        &~ (a \sin \mu t + b \cos \mu t) (c \sin \Omega t + d \cos \Omega t) (e \sin \nu t + f \cos \nu t) \\
        &= \cos((\mu-\nu-\Omega)t) (-bce+ade+acf+bdf)/4 + \cos((\mu+\nu-\Omega)t) (bce-ade+acf+bdf)/4  \\
        &+ \cos((\mu-\nu+\Omega)t) (-bce-ade+acf+bdf)/4 + \cos((\mu+\nu+\Omega)t) (-bce-ade-acf+bdf)/4 \\
        &+ \text{four more sin terms},
    \end{align*}
    and hence we have \eqref{eq_algebraic_representation}.
\end{proof}

\begin{Proposition}[Algebraic calculation of effective matrix] \label{prop:alebraicB}
    Under Condition \ref{cd_setting}, denote $A$'s Jordan blocks by $A_{ii}$. Let $\lambda_i \pm \sqrt{-1} \mu_i$ be the eigenvalue(s) associated to $A_{ii}$. Let $L_{ij}$ be the set of all nonnegative integers $l$ such that $\left| \mu_i \pm \mu_j \right| = l\omega $.
    Then $L_{ij}$ is a finite set, and expressing $B=\mathcal{G}[\exp(-At)P(t)\exp(At)]$ in the same block division as $A$, we have that
    \begin{itemize}
        \item $B_{ij}=0$, if $\lambda_i > \lambda_j$.
        \item $B_{ij}=0$, if $\lambda_i < \lambda_j$ and $P_{ij}=0$.
        \item $\|B_{ij}\|=\infty$, if $\lambda_i < \lambda_j$ and $P_{ij} \neq 0$.
        \item $\|B_{ij}\|=\infty$, if $\lambda_i = \lambda_j$ and the representation of $\exp(-A_{ii} t) P_{ij}(t) \exp(A_{jj} t)$  obtained from  Lemma \ref{lem_fourierExpansion} contains terms in $t^k$  with $k\geq 1$.
        \item $B_{ij}=\sum_{l\in L_{ij}} \bar{M}_{ij,l}$, if $\lambda_i = \lambda_j$ and the representation of $\exp(-A_{ii} t) P_{ij}(t) \exp(A_{jj} t)$  obtained from  Lemma \ref{lem_fourierExpansion} does not contain terms in $t^k$  with $k\geq 1$; $\bar{M}_{ij,l}$ is defined by \eqref{eq_algebraic_representation} in Proposition \ref{thm_algebraic}.
    \end{itemize}
\end{Proposition}

Observe that the presence of terms in $t^k$ with $k\geq 1$ in the representation of $\exp(-A_{ii} t) P_{ij}(t) \exp(A_{jj} t)$  obtained from  Lemma \ref{lem_fourierExpansion} can be checked analytically.  If this representation does not contain such elements, the case $\lambda_i=\lambda_j$  is characterized by only a finite number of Fourier coefficients of $P_{ij}(t)$. Therefore, whether $B$ exists can be checked and its exact expression can be obtained, both in a number of computational steps independent from $\epsilon$.

\subsection{Preparatory analysis}

\begin{Lemma}
    For fixed $a>0,k \in \{0,1,2,\cdots\},b\in \mathbb{R}$ or $a=0,k \in \{1,2,\cdots\},b\in \mathbb{R}$, if $T \gg 0$, the following integrals have asymptotic behavior
    \begin{align*}
        \int_0^T e^{at}t^k \cos(bt)\, dt \sim \frac{+a \cos(bT)+b\sin(bT)}{a^2+b^2} e^{aT} T^k \\
        \int_0^T e^{at}t^k \sin(bt)\, dt \sim \frac{-b \cos(bT)+a\sin(bT)}{a^2+b^2} e^{aT} T^k
    \end{align*}
    in the sense that $f(T)\sim g(T)$ if and only if
    \[
        \lim_{T\rightarrow \infty}\frac{\|f(T)-g(T)\|} {\max \left( \|f(T)\|,\|g(T)\| \right)}=0.
    \]
\label{lem_asymIntegral}
\end{Lemma}

\begin{proof}
    \noindent (i) When $a>0$, recall that the upper incomplete gamma function is defined as
    \[
        \Gamma(s,z)=\int_z^\infty t^{s-1} e^{-t} \, dt.
    \]
    Therefore,
    \begin{align*}
        I&:=\int_0^T e^{at} t^k \cos(bt) \, dt = \frac{1}{2} \Big( \left( (-a-\imath b)^{-1-k}+(-a+\imath b)^{-1-k} \right) (1+k)! \\
        &\qquad - \left( -(a-\imath b)^{-1-k} \Gamma(1+k,-(a-\imath b)T) + -(a+\imath b)^{-1-k} \Gamma(1+k,-(a+\imath b)T) \right) \Big).
    \end{align*}
    Note that $\Gamma(s,z)$, when $s$ fixed, $|z|$ large and $|\arg z|<\frac{3}{2}\pi$, has asymptotic behavior (e.g., \cite{NIST:DLMF})
    \[
        \Gamma(s,z)=z^{s-1}e^{-z}(1+\mathcal{O}(z^{-1})).
    \]
    Therefore,
    \begin{align*}
        I &= \frac{1}{2} \Bigg( \left( (-a-\imath b)^{-1-k}+(-a+\imath b)^{-1-k} \right) (1+k)! + \frac{1}{a-\imath b}T^k e^{(a-\imath b)T}\left(1+\mathcal{O}\left(\frac{1}{T}\right)\right) \\
            & \qquad\qquad\qquad\qquad\qquad\qquad + \frac{1}{a+\imath b}T^k e^{(a+\imath b)T}\left(1+\mathcal{O}\left(\frac{1}{T}\right)\right) \Bigg) \\
        &\sim \frac{1}{2} \left( \frac{1}{a-\imath b}T^k e^{aT}(\cos(bT)-\imath \sin(bT)) + \frac{1}{a+\imath b}T^k e^{aT}(\cos(bT)+\imath \sin(bT)) \right) \\
        &= \frac{+a \cos(bT)+b\sin(bT)}{a^2+b^2} e^{aT} T^k .
    \end{align*}

    \noindent (ii) When $a=0, k \in \{1,2,\cdots\}$, integration by parts gives
    \[
        \int_0^T t^k \cos(bt) \, dt \sim \frac{1}{b}\sin(bT)T^k
    \]
    when $T$ is large. That is, the same expression in (i) works.

    \noindent (iii) A procedure similar to (i) and (ii) shows
    \[
        \int_0^T e^{at}t^k \sin(bt)\, dt \sim \frac{-b \cos(bT)+a\sin(bT)}{a^2+b^2} e^{aT} T^k .
    \]
\end{proof}

\begin{Definition}
    Given $\exp(-A t) P(t) \exp(A t)=\sum_{(a,b,k)\in\Sigma} \left( C_{abk} t^k e^{at} \cos(b t) + t^k D_{abk}  e^{at} \sin(b t) \right)$
    (the  representation of $\exp(-A t) P(t) \exp(A t)$  obtained from  Lemma \ref{lem_fourierExpansion})
    and constant vectors $\Upsilon$ and $\Omega$, define the growth component of $\Upsilon-\exp(-A t) P(t) \exp(A t) \Omega$ as
    \begin{align}
        &~ \mathcal{G}[\Upsilon-\exp(-A t) P(t) \exp(A t) \Omega] \nonumber\\
        :&= \Upsilon-\left( \sum_{(a,b,k)\in\Sigma,a>0}+\sum_{(a,b,k)\in\Sigma,a=0,k\neq 0}+\sum_{(a,b,k)\in\Sigma,a=0,k=0,b=0} \right)  \left( t^k e^{at} \cos(b t) C_{abk} \Omega + t^k e^{at} \sin(b t) D_{abk} \Omega  \right) \nonumber\\
         &= \Upsilon-\mathcal{G}[\exp(-A t) P(t) \exp(A t)] \Omega,
    \end{align}
    \label{def_nonOsc2}
    where $\mathcal{G}[\exp(-A t) P(t) \exp(A t)] $ is defined in \eqref{eq_growthComponent}.
\end{Definition}

\begin{Lemma}
    Given constant vectors $\Upsilon$ and $\Omega$, the solution of ODE
    \begin{equation}
        \dot{y}=-\Upsilon+\exp(-A t) P(t) \exp(A t) \Omega
        \label{avjniurnogiufb1oiop4bu1obpi}
    \end{equation}
    remains bounded if and only if
    \begin{equation}
        \mathcal{G}[\Upsilon-\exp(-A t) P(t) \exp(A t) \Omega]=0.
    \end{equation}
    \label{lem_boundedODE}
\end{Lemma}

\begin{proof}
    By Lemma \ref{lem_fourierExpansion}, we can assume that
    \begin{equation}
        \Upsilon-\exp(-A t) P(t) \exp(A t) \Omega = \sum_{(a,b,k)\in \Sigma_1} t^k e^{at} \cos(b t) c_{abk} + \sum_{(a,b,k)\in \Sigma_2}  t^k e^{at} \sin(b t) d_{abk}
    \end{equation}
    for some sets $\Sigma_1$ and $\Sigma_2$, and nonzero vectors $c_{abk}$, $d_{abk}$. We adopt the convention that $(a,b=0,k)\not\in \Sigma_2$ so that this decomposition is unique.

    Consider the solutions $y^{\cos}_{abk}$ and $y^{\sin}_{abk}$ to
    \begin{eqnarray*}
        \dot{y}^{\cos}_{abk}=- t^k e^{at} \cos(b t) c_{abk} \\
        \dot{y}^{\sin}_{abk}=- t^k e^{at} \sin(b t) d_{abk} .
    \end{eqnarray*}

    Naturally, when $a>0$, the solutions will not remain bounded. When $a=0$ and $k>0$ (recall $k\geq 0$), they will not be bounded either. When $a=0$ and $k=0$, $y^{\cos}_{abk}$ remains bounded if and only if $b\neq 0$, and $y^{\sin}_{abk}$ is bounded for $b\neq 0$ and undefined for $b=0$.

    Note $y(t)=\sum_{(a,b,k)\in \Sigma_1} y^{\cos}_{abk}(t)+\sum_{(a,b,k)\in \Sigma_2} y^{\sin}_{abk}(t)$ is the unique solution to \eqref{avjniurnogiufb1oiop4bu1obpi}. If all $y^{\cos}_{abk}$ and $y^{\sin}_{abk}$ remain bounded, so do $y(t)$; on the other hand, if some $y^{\cos}_{abk}$ and/or $y^{\sin}_{abk}$ $y(t)$ are unbounded, $y(t)$ will be unbounded too, because cancelation will not happen due to different growth rates of $y^{\cos}_{abk}$ and $y^{\sin}_{abk}$.

    Hence, the necessary and sufficient condition for bounded $y$ is $\Sigma_1$ and $\Sigma_2$ being subsets of $\{(a,b,k) | (a<0) \text{ or } (a=0,k=0,b\neq 0) \}$ (note $b=0$ is meaningless for $\Sigma_2$), which by Definition \ref{def_nonOsc2} is equivalent to
    \begin{equation}
        \mathcal{G}[\Upsilon-\exp(-A t) P(t) \exp(A t) \Omega]=0 .
    \end{equation}
\end{proof}

\begin{Lemma}
    Let $R(t)=\exp(-At)P(t)\exp(At)-\mathcal{G}[\exp(-At)P(t)\exp(At)]$. Then there exists a constant $C$ such that $\|R(t)\| \leq C$ for all $t\geq 0$. Furthermore, it has an antiderivative $\mathcal{R}(t)$ (i.e., $\frac{d}{dt} \mathcal{R}(t) = R(t)$) such that $\|\mathcal{R}(t)\| \leq C$ for all $t\geq 0$ too.
    \label{lem_remainderOperator}
\end{Lemma}

\begin{proof}
    By the definition of growth operator
    \begin{align}
        R(t)=\left( \sum_{(a,b,k)\in\Sigma,a=0,k=0,b\neq 0}+\sum_{(a,b,k)\in\Sigma,a<0} \right) (C_{abk}t^k e^{at}\cos(bt)+D_{ab}t^k e^{at}\sin(bt)).
        \label{eq_Rstructure}
    \end{align}
    Since $e^{at} t^k$ is bounded for $a<0$ and $t\geq 0$, and $\cos(bt)$ and $\sin(bt)$ are bounded for real $b$ and $t$, $\|R(t)\| \leq C$ for some constant $C$ for all $t\geq 0$.

    Moreover, when $a=0,k=0,b\neq 0$, antiderivatives of $\cos(bt)$ and $\sin(bt)$ are bounded. As for $a<0,b,k$ terms, we note the indefinite integral of $t^k e^{at}$ converges because (i) the integrand is positive, and (ii) $t \times t^k e^{at} \rightarrow 0$ as $t \rightarrow +\infty$. Therefore, the antiderivative of $t^k e^{at} \cos(bt)$ remains bounded as $t\rightarrow +\infty$, because it is bounded by the indefinite integral of $t^k e^{at}$.

    Therefore, $\|\mathcal{R}(t)\| \leq C$ for all $t\geq 0$ too.
\end{proof}

\subsection{Temporal homogenization}
\label{section_temporalHomogenization}
\paragraph{Heuristic derivation.}
    The intuition behind Theorem \ref{thm_temporalHomogenization} lies in the
    introduction of the 2-scale asymptotic expansion ansatz, popular in perturbation analysis and classical homogenization (see, for instance,  \cite{Na73} or \cite{BeLiPa78}):
    \begin{equation}
        x(t)=x_0(\eta,\xi)+\epsilon x_1(\eta,\xi)+\mathcal{O}(\epsilon^2),
        \label{eq_2scaleExpansion}
    \end{equation}
    where $\eta:=\epsilon t$ and $\xi:=t$ correspond to slow and fast timescales, and are treated as independent variables as $\epsilon\rightarrow 0$; $x_i$'s are such that $\| x_0 \| \gg \epsilon \| x_1 \| \gg \cdots$ for at least $t=\mathcal{O}(\epsilon^{-1})$ as $\epsilon\rightarrow 0$.

    Due to the separation of timescales, formally differential operator $\frac{d}{dt}=\frac{\partial}{\partial \xi}+\epsilon\frac{\partial}{\partial \eta}$. Consequently, \eqref{eq_tempHomoSystem} can be written as
    \begin{equation}
        \frac{\partial x}{\partial \xi}+\epsilon \frac{\partial x}{\partial \eta}=Ax+\epsilon P(\xi)x+f(\xi).
    \end{equation}

    Plot the expansion of $x(t)$ (Eq. \ref{eq_2scaleExpansion}) into the above PDE. Matching $\mathcal{O}(1)$ terms leads to
    \begin{equation}
        \frac{\partial x_0}{\partial \xi}=A x_0+f(\xi),
        \label{eq_leadingOrder}
    \end{equation}
    and matching $\mathcal{O}(\epsilon)$ terms leads to
    \begin{equation}
        \frac{\partial x_1}{\partial \xi} + \frac{\partial x_0}{\partial \eta} = A x_1 + P(\xi) x_0.
        \label{eq_firstOrder}
    \end{equation}

    Solving \eqref{eq_leadingOrder}, we get
    \begin{equation}
        x_0=\exp(A\xi) \left( \Omega(\eta) + \int_0^\xi \exp(-A\tau) f(\tau) \, d\tau \right)
        \label{eq_leadingOrder_sln}
    \end{equation}
    for some vector-valued function $\Omega(\cdot)$.

    Substituting \eqref{eq_leadingOrder_sln} into \eqref{eq_firstOrder}, we obtain
    \begin{equation}
        \frac{\partial x_1}{\partial \xi}+\exp(A\xi)\Omega'(\eta)=A x_1+P(\xi)\exp(A\xi)\Omega(\eta)+P(\xi)\int_0^\xi \exp(A(\xi-\tau))f(\tau)\,d\tau.
    \end{equation}

    Let $y(\xi,\eta):=\exp(-A\xi)x_1(\xi,\eta)$, then we have
    \begin{equation}
        \frac{\partial y}{\partial \xi}=-\Omega'(\eta)+e^{-A \xi}P(\xi)e^{A \xi} \Omega(\eta) + F(\xi) ,
        \label{eq_firstOrder_secular}
    \end{equation}
    where $F(s):=e^{-A s}P(s)e^{A s} \int_0^s e^{-A\tau}f(\tau)\,d\tau$.

    To satisfy $\| x_0 \| \gg \epsilon \| x_1 \|$, we require $y(\xi)$ to be bounded by a constant independent of $\epsilon$. Formally, let
    \begin{equation}
        \bar{F}(\eta) := \epsilon \int_{\eta/\epsilon}^{(\eta+1)/\epsilon} F(\xi) \, d\xi
        \label{eq_formalAverageInput}.
    \end{equation}
    Make a decomposition $y=y_1+y_2$, where
    \begin{align*}
        &\frac{\partial y_1}{\partial \xi}=-\Omega'(\eta)+e^{-A \xi}P(\xi)e^{A \xi} \Omega(\eta) + \bar{F}(\eta) ,\\
        &\frac{\partial y_2}{\partial \xi}=F(\xi) - \bar{F}(\eta) .
    \end{align*}
    Since $\eta$ and $\xi$ are independent variables as $\epsilon\rightarrow 0$, $\bar{F}(\eta)$, $\Omega'(\eta)$ and $\Omega(\eta)$ are viewed as constant vectors at the fast timescale of $\xi$. By definition of $\bar{F}$, $y_2$ is bounded as $\epsilon\rightarrow 0$; at the same time, Lemma \ref{lem_boundedODE} suggests $y_1$ is bounded if and only if $\mathcal{G}[\Omega'(\eta)-\exp(-A \xi) P(\xi) \exp(A \xi) \Omega(\eta)+\bar{F}(\eta)]=0$, which leads to (see Definition \ref{def_nonOsc1}):
    \begin{equation}
        \Omega'(\eta)=\mathcal{G}[\exp(-A \xi) P(\xi) \exp(A \xi)] \Omega(\eta)+\bar{F}(\eta).
    \end{equation}
    When $\exp(-A \xi) P(\xi) \exp(A \xi)$ is bounded, $\mathcal{G}[\exp(-A \xi) P(\xi) \exp(A \xi)]$ is a constant (denoted by $B$), and $\Omega$ is a function of $\eta$ only, consistent with the ansatz of scale separation. Going back to original time variable $t$, the above equation is
    \begin{equation}
        \dot{\Omega}(t)=\epsilon B \Omega(t) + \epsilon \bar{F}(\epsilon t).
        \label{eq_cellProblemOld}
    \end{equation}
    However, one problem remains: does the right side of \eqref{eq_formalAverageInput} have a limit as $\epsilon\rightarrow 0$? Rather than imposing extra restrictions on $f$ (such as it is fast/slow), we prefer a general result, and heuristically replace the cell problem \eqref{eq_cellProblemOld} by
    \begin{equation}
        \dot{\Omega}(t)=\epsilon B \Omega(t) + \epsilon F(t).
        \label{eq_cellProblem}
    \end{equation}
    We then prove the effective solution \eqref{eq_leadingOrder_sln} generated by this $\Omega$ still has small error.


\paragraph{Rigorous justification.}
\begin{proof}[Proof of Theorem \ref{thm_temporalHomogenization}]~\\
    Let $\Xi(t)=\exp(-A t)x(t)-\int_0^t \exp(-A\tau) f(\tau)$, then
    \begin{align*}
        \dot{\Xi}(t) &= \epsilon \exp(-At)P(t) \exp(At) \left( \Xi(t) + \int_0^t \exp(-A\tau) f(\tau) \,d\tau \right) \\
        &= \epsilon \exp(-At)P(t) \exp(At) \Xi(t) + \epsilon F(t), \qquad\qquad
        \Xi(0) = x(0).
    \end{align*}
    Since $\exp(-At)P(t) \exp(At)$ is bounded, $\mathcal{G}[\exp(-At)P(t) \exp(At)]$ is a constant. Let it be $B$.
    Consider
    \begin{align*}
        \dot{\Omega}(t) &= \epsilon B \Omega(t) + \epsilon F(t), \qquad\qquad
        \Omega(0) = x(0).
    \end{align*}
    Let $E(t)=\Xi(t)-\Omega(t)$ and $R(t)=\exp(-At)P(t)\exp(At)-B$. Then
    \begin{align*}
        \dot{E}(t) &= \epsilon B E(t) + \epsilon R(t) \Xi(t), \qquad\qquad
        E(0) = 0.
    \end{align*}
    Let $P(t)=\exp(\epsilon B t)$, then
    \begin{align}
        E(t) & =P(t) E(0) + P(t) \int_0^t P(\tau)^{-1} \epsilon R(\tau) \Xi(\tau) \, d\tau \nonumber\\
        &=\epsilon P(t) \int_0^t P(\tau)^{-1} R(\tau) (E(\tau)+\Omega(\tau)) \, d\tau \nonumber\\
        &=\epsilon P(t) \int_0^t P(\tau)^{-1} R(\tau) E(\tau) \, d\tau + I(t),
    \end{align}
    where
    \[
        I(t):=\epsilon P(t) \int_0^t P(\tau)^{-1} R(\tau) \Omega(\tau) \, d\tau.
    \]
    Treat $t$ as fixed for now and let $\mathcal{P}(\tau)=P(t)P(\tau)^{-1}$. Then
    \[
        \mathcal{P}'=-P(t)P(\tau)^{-1} P'(\tau) P(\tau)^{-1}=-P(t)P(\tau)^{-1}\epsilon B P(\tau)P(\tau)^{-1}
        =-\epsilon \mathcal{P} B,
    \]
    where prime means derivative with respect to $\tau$.

    Let $\mathcal{R}$ be the antiderivative of $R$ defined in Lemma \ref{lem_remainderOperator}. Then
    \begin{align*}
        I(t)
        &= \epsilon \int_0^t \mathcal{P}(\tau) R(\tau) \Omega(\tau) \, d\tau \\
        &= \epsilon \int_0^t \mathcal{P}(\tau) d \mathcal{R} (\tau) \Omega(\tau) \\
        &= \epsilon \mathcal{P}(t) \mathcal{R}(t) \Omega(t) - \epsilon \mathcal{P}(0) \mathcal{R}(0) \Omega(0) - \epsilon \int_0^t \mathcal{P}'(\tau) \mathcal{R}(\tau) \Omega(\tau) \, d\tau - \epsilon \int_0^t \mathcal{P}(\tau) \mathcal{R}(\tau) \Omega'(\tau) \, d\tau \\
        &= \epsilon \mathcal{R}(t) \Omega(t) - \epsilon \exp(\epsilon B t) \mathcal{R}(0) \Omega(0) + \epsilon^2 \int_0^t \mathcal{P}(\tau) B \mathcal{R}(\tau) \Omega(\tau) \, d\tau - \epsilon^2 \int_0^t \mathcal{P}(\tau) \mathcal{R}(\tau) (B \Omega(\tau) + F(\tau)) \, d\tau.
    \end{align*}
    Note $B$, $\exp(\epsilon B t)$, and hence $\mathcal{P}(t)$ all remain bounded till $t=\mathcal{O}(\epsilon^{-1})$. Also, $\mathcal{R}(t)$ remains bounded for all time by Lemma \ref{lem_remainderOperator}. Therefore,
    \[
        \| I(t) \| \leq C_1 \epsilon \max_{\tau\in [0,t]} \|\Omega(\tau)\| + C_2 \epsilon \max_{\tau\in [0,t]} \left\| F(\tau) \right\|
    \]
    for $|t|\leq C_3\epsilon^{-1}$ and some constants $C_1,C_2,C_3>0$. Since
    \[
        F(t)=\exp(-At)P(t)\exp(At)\int_0^t \exp(-A\tau)f(\tau)\,d\tau
    \]
    and $\exp(-At)P(t)\exp(At)$ is bounded by assumption, there is some $C>0$ such that
    \[
        \| I(t) \| \leq C\epsilon \left( \max_{\tau\in [0,t]} \|\Omega(\tau)\| + \max_{\tau\in [0,t]} \left\| \int_0^\tau \exp(-A s)f(s)\,ds \right\| \right).
    \]
    Similarly, we have
    \begin{align}
        E(t) &=\epsilon \int_0^t \mathcal{P}(\tau) R(\tau) E(\tau) \, d\tau + I(t) \nonumber\\
        &= \epsilon \mathcal{P}(t) \mathcal{R}(t) E(t)-\epsilon \mathcal{P}(0) \mathcal{R}(0) E(0) - \epsilon \int_0^t \mathcal{P}' \mathcal{R} E \, d\tau - \epsilon \int_0^t \mathcal{P} \mathcal{R} E' \, d\tau + I(t) \nonumber\\
        &= \epsilon \mathcal{R}(t) E(t) + \epsilon \int_0^t \epsilon \mathcal{P} B \mathcal{R} E \, d\tau - \epsilon \int_0^t \mathcal{P} \mathcal{R} \epsilon (B E+R E+R \Omega) \, d\tau + I(t).
        \label{eq_ErrorIntegralRepresentation}
    \end{align}
    Let $J(t):=-\epsilon \int_0^t \mathcal{P} R \epsilon R \Omega \, d\tau+I(t)$. It can be analogously shown that
    \[
        \| J(t) \| \leq C\epsilon \left( \max_{\tau\in [0,t]} \|\Omega(\tau)\| + \max_{\tau\in [0,t]} \left\| \int_0^\tau \exp(-A s)f(s)\,ds \right\| \right)
    \]
    for some $C>0$. Rearranging terms in \eqref{eq_ErrorIntegralRepresentation}, we obtain
    \[
        E(t)=(1-\epsilon \mathcal{R}(t))^{-1} \left( \epsilon^2 \int_0^t (\mathcal{P}B\mathcal{R}-\mathcal{P}\mathcal{R}B-\mathcal{P}\mathcal{R} R)(\tau) E(\tau) \,d\tau + J(t) \right).
    \]
    Let $e(t)=\|E(t)\|$. Since $(\mathcal{P}B\mathcal{R}-\mathcal{P}\mathcal{R}B-\mathcal{P}\mathcal{R} R)(\tau)$ remains bounded till at least $t=\mathcal{O}(\epsilon^{-1})$, we have
    \[
        e(t) \leq \epsilon^2 \int_0^t C e(\tau) \, d\tau + C \epsilon \left( \max_{\tau\in [0,t]} \|\Omega(\tau)\| + \max_{\tau\in [0,t]} \left\| \int_0^\tau \exp(-A s)f(s)\,ds \right\| \right).
    \]
    Gronwall's inequality gives
    \[
        e(t) \leq \exp(\epsilon^2 C t) C \epsilon \left( \max_{\tau\in [0,t]} \|\Omega(\tau)\| + \max_{\tau\in [0,t]} \left\| \int_0^\tau \exp(-A s)f(s)\,ds \right\| \right).
    \]
\end{proof}

\begin{Remark}
     The relative error is quantified in \eqref{eq_errorBound} by comparing the absolute error with the approximated solution after an appropriate scaling.
\end{Remark}

\begin{Remark}
    The inhomogeneous term $f(\cdot)$ may not be small nor periodic. When it is, it can be homogenized. This can be done in our framework by concatenating $x$ with an extra dummy variable $z$, with $z(0)=1$, $\dot{z}=0$, and $f(t)$ replaced by $f(t)z$.
\end{Remark}

The following corollary shows that, in the homogeneous case, one can drop the $\epsilon P(t) x$ term in \eqref{eq_tempHomoSystem} without loss of accuracy if $P(t)$ does not oscillate at a resonant frequency (defined as the difference between the imaginary parts of two eigenvalues of $A$). Unlike Theorems \ref{thm_temporalHomogenization} and \ref{thm_tempHomoCannot}, this is only a sufficient condition.

\begin{Corollary}
    Consider system \eqref{eq_tempHomoSystem}. Assume without loss of generality that the Fourier expansion of $P(t)$ does not contain constant terms (such terms can be absorbed into $A$), and denote by $2\pi/\omega$ the smallest period of $P(t)$. Suppose $f(t)\equiv 0$. Assume that $A$ is diagonalizable and that all its eigenvalues (indicated by $\lambda_i+\sqrt{-1}\mu_i$) have the same real part (i.e., $\lambda_i=\lambda$ for all $i$)\footnote{An example is a mechanical system subject to isotropic dissipation and with bounded trajectory.}. If there is no integer $l$ such that
    \begin{equation}
        | \mu_i \pm \mu_j | = l \omega
    \end{equation}
    for some $i,j \in \{1,\ldots, m\}$, then
    \[
        x(t)=\exp(At)(x(0)+E(t,\epsilon)),
    \]
    with
    \begin{equation}
         \| E(t, \epsilon) \| \leq C \epsilon \exp(\epsilon^2 C t),
    \end{equation}
    for some constant $C$ independent from $t$ and $\epsilon$ when $t\leq C\epsilon^{-1}$.
\end{Corollary}

\begin{proof}
    Proposition \ref{thm_algebraic} shows that $B=\mathcal{G}[\exp(-At)P(t)\exp(At)]=0$. Then apply Theorem \ref{thm_temporalHomogenization}.
\end{proof}


\begin{proof}[Proof of Theorem \ref{thm_tempHomoCannot}]~\\
    Let $G(t):=\exp(-At) P(t) \exp(At)$. Since $G(t)$ is unbounded in $t$, when written in canonical form (Lemma \ref{lem_fourierExpansion}), it
    contains at least one $e^{at}t^k \cos(bt)$ or $e^{at}t^k \sin(bt)$ term with either $a>0$ or $(a=0,k>0)$. Choose $a,k,b$ that correspond to the fastest growing term. The proof is by contradiction:

    Suppose there exists a constant matrix $B$, independent of the choice of $f$, such that for all initial condition $x_0$ and all $t \leq \bar{C} \epsilon^{-1}$ for some $\bar{C}$,
    \[
        \|E(t,\epsilon)\| \leq C \epsilon \left( \max_{\tau\in [0,t]} \|\Omega(\tau)\| + \max_{\tau\in [0,t]} \left\| \int_0^\tau \exp(-A s)f(s)\,ds \right\| \right)
    \]
    for some $C$. Then the above should hold for a particular choice of $f\equiv 0$. In this case,
    \[
        \Omega(t)=\exp(\epsilon B t)x(0),
    \]
    and therefore as long as $t \leq \bar{C} \epsilon^{-1}$,
    \[
        \|E(t,\epsilon)\| \leq C\epsilon.
    \]
    As before, we have
    \begin{equation}
        \dot{E}(t)=\epsilon B E(t)+\epsilon (G(t)-B) \Xi(t) , \qquad\qquad E(0)=0 ,
        \label{eq_scaledError}
    \end{equation}
    where $\Xi(t)=\exp(-At)x(t)$ satisfies $\dot{\Xi}(t)=\epsilon G(t) \Xi(t)$ and $\Xi(0)=x_0$.

    Variation of constants leads to
    \[
        E(t)=\int_0^t \exp(\epsilon B (t-\tau)) \epsilon (G(\tau)-B) (\Omega(\tau)+E(\tau)) \, d\tau.
    \]
    Rearranging terms, we have
    \begin{equation}
        \epsilon \int_0^t e^{\epsilon B(t-\tau)} G \Omega \, d\tau=E(t)+\epsilon\int_0^t e^{\epsilon B(t-\tau)} B(\Omega+E)\,d\tau-\epsilon\int_0^t e^{\epsilon B(t-\tau)} G E\,d\tau.
        \label{eqn_errorIntegralForm}
    \end{equation}
    Assume without loss of generality $\bar{C}=1$ and choose $t=\epsilon^{-1}$, then right hand side (RHS) of \eqref{eqn_errorIntegralForm} satisfies
    \begin{align*}
        \|\text{RHS of \eqref{eqn_errorIntegralForm}}\| &\leq C\epsilon + \epsilon\int_0^{\epsilon^{-1}} C \cdot C \cdot (C+C\epsilon) \,d\tau + \epsilon \int_0^{\epsilon^{-1}} C \max_{0 \leq s \leq \tau} \|G(s)\| C \epsilon \,d\tau \\
        &\leq C \epsilon^2 \int_0^{\epsilon^{-1}} e^{a\tau}\tau^k \, d\tau.
    \end{align*}
    Lemma \ref{lem_asymIntegral} leads to
    \[
        \|\text{RHS of \eqref{eqn_errorIntegralForm}}\| \leq C \epsilon^2 e^{a/\epsilon} (\epsilon^{-1})^k.
    \]
    On the other hand, the left hand side (LHS) of \eqref{eqn_errorIntegralForm} is
    \[
        \text{LHS of \eqref{eqn_errorIntegralForm}}=\epsilon \int_0^t e^{\epsilon B (t-\tau)} G(\tau) e^{\epsilon B \tau}x(0) \,d\tau.
    \]
    Write $B$ in Jordan canonical form $B=V^{-1}J V$, where
    \[
        J=\begin{bmatrix} \lambda_1 & d_1 & & \\ & \lambda_2 & \ddots & \\ & & \ddots & d_{n-1} \\ & & & \lambda_n \end{bmatrix},
    \]
    $\lambda$'s are $B$ eigenvalues, superdiagonal elements $d$'s are either 0 or 1, and $V$ is orthonormal. Then
    \[
        \text{LHS of \eqref{eqn_errorIntegralForm}}=\epsilon V^{-1} \int_0^t e^{\epsilon J (t-\tau)} V G(\tau) V^{-1} e^{\epsilon J \tau} V x(0) \, d\tau.
    \]
    Let $\bar{G}(\tau)=V G(\tau) V^{-1}$. Since the conjugate transform preserves the matrix norm, $G(t)$, when written in canonical form, still has at least one element that contains an $e^{at}t^k \cos(bt)$ or $e^{at}t^k \sin(bt)$ term. Because LHS of \eqref{eqn_errorIntegralForm} is a linear functional of $G(\cdot)$, assume without loss of generality that
    \[
        \bar{G}_{ij}(\tau)=e^{a\tau}\tau^k \cos(b\tau)
    \]
    for some $i,j\in\{1,2,\cdots,n\}$ (the $e^{a\tau}\tau^k \sin(b\tau)$ case is completely analogous).
    Also, let $y(0)=V x(0)$, then
    \[
        \text{LHS of \eqref{eqn_errorIntegralForm}}=\epsilon V^{-1} \int_0^t e^{\epsilon J (t-\tau)} \bar{G}(\tau) e^{\epsilon J \tau} y(0) \, d\tau .
    \]
    For notational convenience, let
    \[
        L=\int_0^t e^{\epsilon J (t-\tau)} \bar{G}(\tau) e^{\epsilon J \tau} \, d\tau.
    \]
    Suppose $\lambda_i$ and $\lambda_j$ are respectively located in $J$ in $m_1$-by-$m_1$ and $m_2$-by-$m_2$ Jordan diagonal blocks
    \[
        J_1=\begin{bmatrix} \lambda_i & 1 & & \\ & \lambda_i & \ddots & \\ & & \ddots & 1 \\ & & & \lambda_i \end{bmatrix}   \qquad \text{and} \qquad
        J_2=\begin{bmatrix} \lambda_j & 1 & & \\ & \lambda_j & \ddots & \\ & & \ddots & 1 \\ & & & \lambda_j \end{bmatrix}.
    \]
    Isolate the corresponding $m_1$-by-$m_2$ blocks in $L$ and $\bar{G}$ and call them $\hat{L}$ and $\hat{G}$. Then
    \begin{align*}
        \hat{L} &= \int_0^t e^{\epsilon J_1 (t-\tau)} \hat{G}(\tau) e^{\epsilon J_2 \tau} \, d\tau \\
        &= \int_0^t e^{\epsilon \lambda_i(t-\tau)+\epsilon \lambda_j \tau}
        \begin{bmatrix} 1 & \epsilon(t-\tau) & \cdots & \frac{(\epsilon(t-\tau))^{m_1-1}}{(m_1-1)!} \\
                          & 1                & \ddots & \vdots \\
                          &                  & \ddots & \epsilon(t-\tau) \\
                          &                  &        & 1 \end{bmatrix}
        \hat{G}(\tau)
        \begin{bmatrix} 1 & \epsilon \tau & \cdots & \frac{(\epsilon \tau)^{m_2-1}}{(m_2-1)!} \\
                          & 1             & \ddots & \vdots \\
                          &               & \ddots & \epsilon \tau \\
                          &               &        & 1 \end{bmatrix} \,d\tau.
    \end{align*}
    Let $\hat{G}_{\alpha\beta}(\tau)$ be the new location of $\bar{G}_{ij}(\tau)=e^{a\tau}\tau^k \cos(b\tau)$ in submatrix $\hat{G}$. Consider
    \[
        \begin{bmatrix} u_1 & \cdots & u_m \end{bmatrix} = \begin{bmatrix} \hat{G}_{\alpha 1} & \cdots & \hat{G}_{\alpha \beta} & \cdots & \hat{G}_{\alpha m} \end{bmatrix}
        \begin{bmatrix} 1 & \epsilon \tau & \cdots & \frac{(\epsilon \tau)^{m_2-1}}{(m_2-1)!} \\
                          & 1             & \ddots & \vdots \\
                          &               & \ddots & \epsilon \tau \\
                          &               &        & 1 \end{bmatrix},
    \]
    then
    \[
        u_\beta = \hat{G}_{\alpha\beta}+\sum_{i=1}^{\beta-1} \hat{G}_{\alpha i} \frac{(\epsilon \tau)^{\beta-i}}{(\beta-i)!}.
    \]
    Either $u_\beta(\epsilon^{-1})$ is still at the order of $e^{a/\epsilon}(\epsilon^{-1})^k$ (in $\epsilon$) as $\hat{G}_{\alpha\beta}(\epsilon^{-1})$ is, or some later term $\hat{G}_{\alpha i} \frac{(\epsilon \tau)^{\beta-i}}{(\beta-i)!}$ cancels this leading order.

    If the latter case (cancellation), because $\epsilon\tau=\mathcal{O}(1)$, $\hat{G}_{\alpha i}$ must be at this leading order too. In this case, choose a new $\beta$ to be $i$, and repeat the above procedure.

    Because $1\leq i <\beta$ is always true, this procedure will terminate eventually. In the end, there will be some $\beta \in \{1,\cdots,m\}$ such that $u_\beta$ is at the order of $e^{a/\epsilon}(\epsilon^{-1})^k$.

    Now, pick $m_2$-dimensional vector $\hat{y}(0)=\begin{bmatrix} 0 & \cdots & 0 & 1 & 0 & \cdots & 0 \end{bmatrix}$, where the only nonzero element is in column $\beta$. Pick $y(0)$ by padding $\hat{y}(0)$ with 0 elements, such that  the location of $\hat{y}(0)$ in $y(0)$ corresponds to the location of $J_2$ in $J$. If we introduce notation
    \[
        \begin{bmatrix} v_1 \\ v_2 \\ \ldots \\ v_{m_1} \end{bmatrix} = \hat{G} \begin{bmatrix} 1 & \epsilon \tau & \cdots & \frac{(\epsilon \tau)^{m_2-1}}{(m_2-1)!} \\
                          & 1             & \ddots & \vdots \\
                          &               & \ddots & \epsilon \tau \\
                          &               &        & 1 \end{bmatrix} \hat{y}(0),
    \]
    then $v_\alpha = \begin{bmatrix} u_1 & \cdots & u_{m_2} \end{bmatrix} \hat{y}(0) = u_\beta$.

    Using the upper triangular matrix structure again, an analogous argument shows
    \[
        \begin{bmatrix} 1 & \epsilon(t-\tau) & \cdots & \frac{(\epsilon(t-\tau))^{m_1-1}}{(m_1-1)!} \\
                          & 1                & \ddots & \vdots \\
                          &                  & \ddots & \epsilon(t-\tau) \\
                          &                  &        & 1 \end{bmatrix}
        \begin{bmatrix} v_1 \\ v_2 \\ \ldots \\ v_{m_1} \end{bmatrix}
    \]
    also contains an element at the leading order of $e^{a/\epsilon}(\epsilon^{-1})^k$. Lemma \ref{lem_asymIntegral} implies $\hat{L} \hat{y}(0)$ also contains an element at this leading order (up to a constant prefactor due to the $e^{\epsilon \lambda_i(t-\tau)+\epsilon \lambda_j\tau}$ factor involved in $\epsilon^{-1}$ time integral), and therefore so does $L y(0)$.

    Since $V^{-1}$ is orthonormal and hence vector-norm preserving,
    \[
        \| \text{LHS of \eqref{eqn_errorIntegralForm}} \|=\| \epsilon V^{-1} L y(0) \|=\epsilon \| L y(0) \|,
    \]
    and it is at least at the order of $\epsilon e^{a/\epsilon}(\epsilon^{-1})^k$. Since
    \[
        \epsilon e^{a/\epsilon}(\epsilon^{-1})^k \gg C \epsilon^2 e^{a/\epsilon}(\epsilon^{-1})^k > \| \text{RHS of \eqref{eqn_errorIntegralForm}} \|,
    \]
    when $\epsilon$ is small enough, \eqref{eqn_errorIntegralForm} cannot be an equality. This is a contradiction, and hence $B$ does not exist.
\end{proof}

\section{Application 1: Control via parametric resonance}
\label{Section_PR}
\subsection{Parametric resonance in a variant of Mathieu's equation}
\label{Section_PR_sln}
Consider the system
\begin{equation}
    \ddot{x}+\omega^2(1+\epsilon \cos(2\omega t+\theta))x=0 .
    \label{eq_PR_canonical}
\end{equation}
Without the additional phase $\theta$, this is Mathieu's equation, which is known to correspond to parametric resonance (PR for short; see \cite{MathieuEquation, HillEquation, FloquetTheory} for seminal discussions on Mathieu's equation and its generalization known as Hill's equation, with motivations in celestial mechanics; see also \cite{YaSt75b,LaLi76,MaWi04} for some more modern reviews).

This system corresponds to the canonical form \eqref{eq_tempHomoSystem} with
\[
    A=\begin{bmatrix} 0 & 1 \\ -\omega^2 & 0 \end{bmatrix}, \qquad
    P(t)=\begin{bmatrix} 0 & 0 \\ -\omega^2 \cos(2 \omega t+\theta) & 0 \end{bmatrix}, \qquad
    f(t)=\begin{bmatrix} 0 \\ 0 \end{bmatrix}.
\]
A direct computation gives
\[
    B=\mathcal{G}[\exp(-A t) P(t) \exp(A t)]=-\frac{1}{4}\begin{bmatrix} \omega \sin\theta & \cos\theta \\ \omega^2 \cos\theta & -\omega\sin\theta \end{bmatrix}.
\]
This matrix has $0$ trace and $-\omega^2/16$ determinant, and therefore
\[
    \exp(\epsilon B t)=\begin{bmatrix} \cosh \frac{\omega \epsilon t}{4}-\sin\theta \sinh \frac{\omega \epsilon t}{4} & -\cos\theta\sinh \frac{\omega \epsilon t}{4}/\omega \\ -\omega \cos\theta \sinh \frac{\omega \epsilon t}{4} & \cosh \frac{\omega \epsilon t}{4}+\sin\theta\sinh \frac{\omega \epsilon t}{4} \end{bmatrix}.
\]
Hence, we have, till at least $t=\mathcal{O}(\epsilon^{-1})$,
\begin{align}
    x(t) &=\left[ x(0) \cosh \left( \frac{\epsilon \omega}{4}t \right)- \left(\frac{\dot{x}(0)}{\omega} \cos\theta+x(0)\sin\theta \right) \sinh \left( \frac{\epsilon \omega}{4}t \right) \right] \cos\omega t \nonumber\\
    &+ \left[ \frac{\dot{x}(0)}{\omega} \cosh \left( \frac{\epsilon \omega}{4}t \right)- \left(x(0) \cos\theta-\frac{\dot{x}(0)}{\omega}\sin\theta \right) \sinh \left( \frac{\epsilon \omega}{4}t \right) \right] \sin\omega t + \mathcal{O}(\epsilon).
    \label{eq_PRsolution}
\end{align}

\begin{Corollary}[Exponential decay]
    When $\tan \frac{\theta}{2}=\frac{x(0)-\dot{x}(0)/\omega}{x(0)+\dot{x}(0)/\omega}$,
    \[
        x(t)=\exp(-\epsilon\omega t/4)(x(0) \cos(\omega t)+x'(0) \sin(\omega t)/\omega)+\mathcal{O}(\epsilon)
    \]
    till at least $t=\mathcal{O}(\epsilon^{-1})$.
    \label{cor_ExponentialDecay}
\end{Corollary}

\begin{proof}
    Since $\cosh(x)\equiv \exp(x)/2+\exp(-x)/2$ and $\sinh(x)\equiv \exp(x)/2-\exp(-x)/2$ for all $x$, it suffices to show the equivalency of $\tan \frac{\theta}{2}=\frac{x(0)-\dot{x}(0)/\omega}{x(0)+\dot{x}(0)/\omega}$, $x(0)=\dot{x}(0)/\omega \cos\theta+ x(0) \sin\theta$ and $\dot{x}(0)/\omega=x(0)\cos\theta-\dot{x}(0)/\omega\sin\theta$.

    This is immediate upon using basic trigonometric identities $1=\cos^2\frac{\theta}{2}+\sin^2\frac{\theta}{2}$, $\cos\theta=\cos^2\frac{\theta}{2}-\sin^2\frac{\theta}{2}$ and $\sin\theta=2\sin\frac{\theta}{2}\cos\frac{\theta}{2}$.
\end{proof}

\begin{Remark}
    Although parametric resonance oftentimes leads to exponentially growing oscillations, it may, as observed in \cite{LaLi76},  also lead to exponentially decaying solutions. For a  2-dimensional periodic linear ODE system \eqref{eq_PR_canonical} with trace-free time-averaged coefficient matrix,
 Floquet theory (see for instance \cite{Verhulst:96}) guarantees that exponentially growing and decaying solutions always come in pairs.
    Corollary \ref{cor_ExponentialDecay}  shows how to obtain this decaying solution. Note that the decay can either be achieved by a careful choice of initial condition (such that $x(0)=\dot{x}(0)/\omega$), or  by adding a phase in the perturbation to adjust to arbitrary initial condition.
\end{Remark}

\begin{Remark}
    For $\theta$ such that $\tan \frac{\theta}{2} \neq \frac{x(0)-\dot{x}(0)/\omega}{x(0)+\dot{x}(0)/\omega}$, when $t$ is large $x(t)$ will be dominated by exponentially growth. However, when $\theta/2=\arctan \frac{a-b}{a+b}+\mathcal{O}(\epsilon)$, it can be shown that $x(t)$ decays when $t$ is not too large; this is why the proposed method of control (see  Section \ref{Section_control}) is robust to small perturbations in $\theta$ caused by implementation errors.
\end{Remark}

\subsection{Control of oscillations}
\label{Section_control}

Given a smooth enough, positive-valued function $f(t)$, our purpose is to control the amplitude of the oscillations of the solution of
\begin{equation}\label{eq:omeha0}
    \ddot{x}+\omega^2\left(1+\epsilon \cos(2 \omega t+\theta)\right)x=0
\end{equation}
so that it follows approximately $f(t)$. We will achieve this control by changing the values of $\epsilon$ and $\theta$ over a finite number of time intervals.

\paragraph{Assumption:} We will assume that $f$ is slowly varying when compared to the time scale $0<1/\omega<\infty$, i.e., that
 $f(t)\in \mathcal{C}^1([0,T])$ and
\begin{equation}
    \left| \frac{1}{\omega} \frac{d}{dt}\log f(t) \right| \ll 1
    \label{eq_controlCondition1}
\end{equation}
and
\begin{equation}
    \left| \frac{1}{\omega} \frac{d}{dt} f(t) \right| \ll 1
    \label{eq_controlCondition2}
\end{equation}
for all $t \in [0,T]$, where $T$ is the end time of the control.

The following algorithm describes how the solution of \eqref{eq:omeha0} can be controlled by changing values of $\epsilon$ and $\theta$ on time intervals of length $H$.

\begin{Algorithm}[Control of oscillations by parametric resonance]~\\
\begin{itemize}
\item
    Let $H:=M/\omega$, where $M$ is a pre-set $\mathcal{O}(1)$ constant ($M=2$ in this paper).
\item
    At each time step, i.e., $t=nH$ for  $n\in \mathbb{N}$, compute $r:= f(t+H) \Big/ \sqrt{x(t)^2+\dot{x}(t)^2/\omega^2}$; Let $a=x(t)$ and $b=\dot{x}(t)/\omega$.
\item
    If $r\geq 1$, let $\epsilon=\frac{\log(r)}{\omega H}$ and $\theta=2 \arctan \frac{a+b}{b-a}$ for $t\in [nH,(n+1)H)$.
\item
    If $r\leq 1$, let $\epsilon=-\frac{\log(r)}{\omega H}$ and $\theta=2 \arctan \frac{a-b}{a+b}$ for $t\in [nH,(n+1)H)$.
\item
    $n\rightarrow n+1$ and iterate until $n=\lfloor T/H \rfloor$.
\end{itemize}
\label{alg_control}
\end{Algorithm}

This algorithm works in the sense that it leads to a solution $x(t)$ such that $\sqrt{x(t)^2+\dot{x}(t)^2/\omega^2} \approx f(t)$ for all $t \in [0,T]$. The idea is to approximate $f(\cdot)$ by a piecewise-exponential function with piece-width $H$.

The condition given by \eqref{eq_controlCondition2} ensures that $f(\cdot)$ changes very little within a step of length $H$, so that it is well approximated by a piecewise-linear function with piece-width $H$.

The condition given by \eqref{eq_controlCondition1} leads to
\[
    \left| \frac{\log f(t+H)-\log f(t)}{\omega H} \right| \ll 1 .
\]
That is, if $f(t+H)/f(t)=\exp(\epsilon \omega H/4)$, then $\epsilon\ll 1$. Therefore, as Corollary \ref{cor_ExponentialDecay} shows, the choice of $\theta$ in the algorithm enables a decrease of oscillation amplitude from $\approx f(nH)$ at step $n$ to $\approx f((n+1)H)$ at step $n+1$ (or increase by an analogous reason). Furthermore, since $\epsilon \omega H/4=\epsilon M/4 \ll 1$, the envelope of $f(nH) \exp(\epsilon \omega (\tau-nH)/4), \tau \in [nH,(n+1)H]$ is close to a piecewise-linear approximation of $f(\tau), \tau\in [nH,(n+1)H]$.

In addition, since we use $r= f(t+H) \Big/ \sqrt{x(t)^2+\dot{x}(t)^2/\omega^2}$ but not $r= f(t+H) / f(t)$, the approximation error from the previous step will not affect the current step.

\begin{Remark}
    Conditions  \eqref{eq_controlCondition1} and \eqref{eq_controlCondition2} can be satisfied by choosing $\omega$ large enough, as long as $\log f(t)$ is $\mathcal{C}^1$. This is due to the extreme value theorem and the compactness of $[0,T]$. That is to say, as long as the desired signal is differentiable, it can be approximated by the envelope of high (enough) frequency oscillations.
\end{Remark}

\paragraph{Numerical illustration:}
We arbitrarily chose a function $f(t)=(t-6)(t-5)(t-4)(t-3)(t+0.1)+10$ to demonstrate Algorithm \ref{alg_control}. $f$ is chosen to be a high degree polynomial so that its graph is nontrivial, and the constant is chosen such that $f(t)>0$ for all $t>0$.

\begin{figure}[htb]
\centering
\footnotesize
\subfigure[$\omega=200$]{
\includegraphics[width=0.48\textwidth]{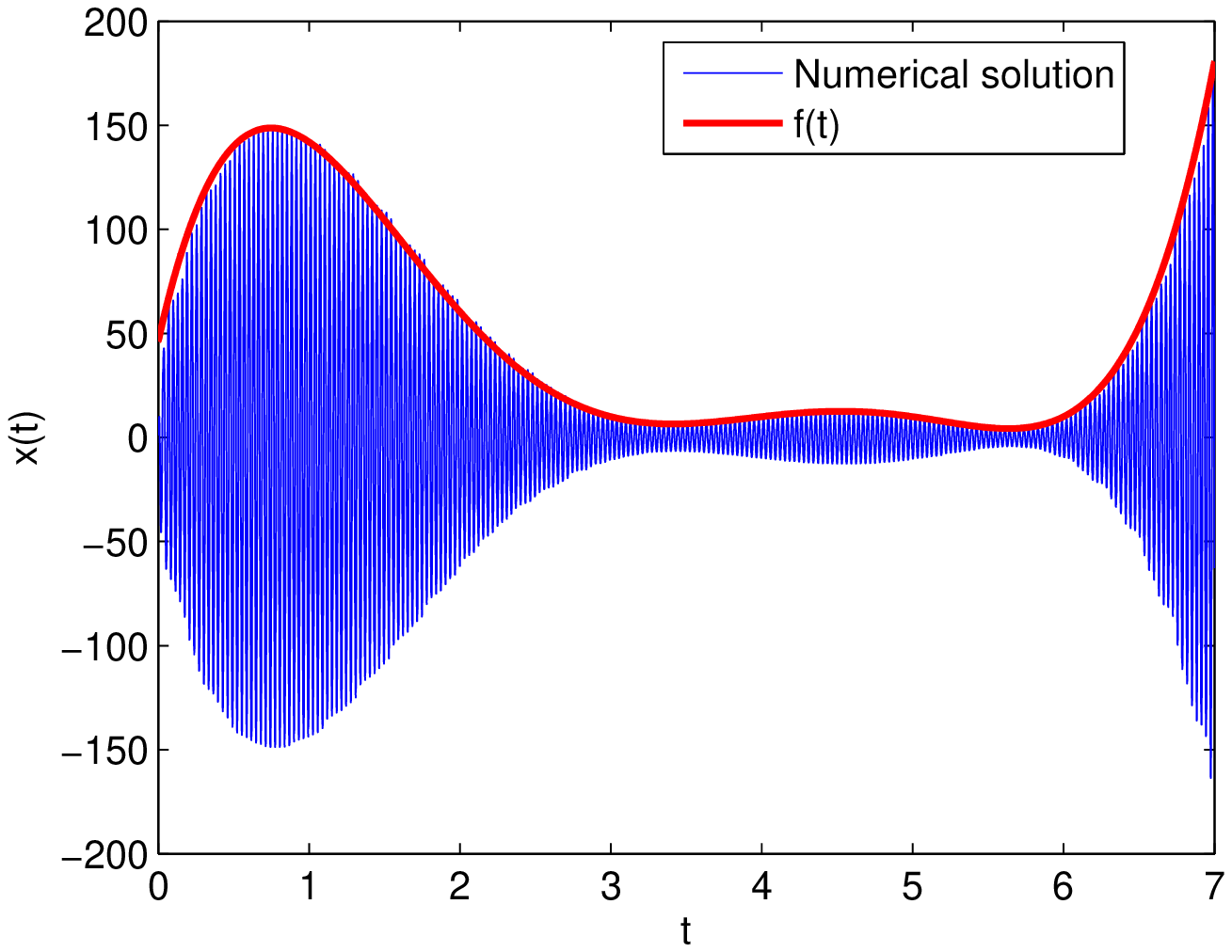}
\label{fig_5thPoly_a}
}
\subfigure[$\omega=1000$]{
\includegraphics[width=0.48\textwidth]{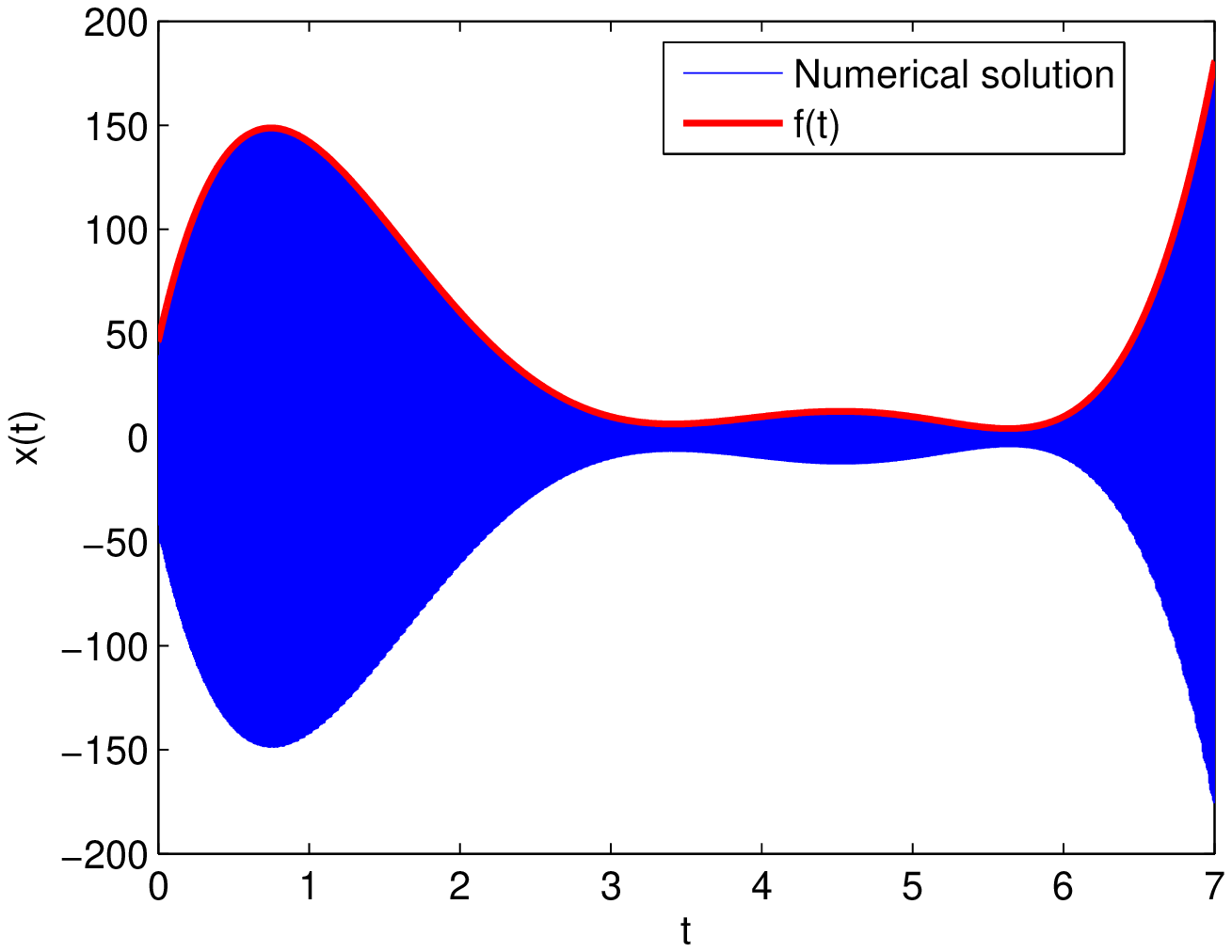}
}
\caption{\footnotesize $x(t)$, the solution of the canonical equation with $\epsilon$ and $\theta$ chosen by the algorithm proposed in Section \ref{Section_control}, compared with the graph of $f(t)$. $x(t)$ is obtained obtained numerically by Velocity-Verlet with timestep $0.1/\omega$.}
\label{fig_5thPoly}
\end{figure}

As can be seen from Figure \ref{fig_5thPoly}, control is achieved in the sense that the oscillation amplitude of $x(t)$ approximates $f(t)$ when $\omega$ is big enough. The initial condition is $x(0)=1$ and $\dot{x}(0)=0$. Even though $f(0)=46$ significantly differs from $x(0)$, the amplitude $\sqrt{x(t)^2+\dot{x}(t)^2/\omega^2}$ rapidly converges to $f(t)$ (at rate $\sim 1/\omega$, and therefore barely observable in Figure \ref{fig_5thPoly_a}). Naturally, larger $\omega$ (and hence smaller $\epsilon$) leads to more accurate match. Longer simulation times do not degrade the quality of the approximation; however they obscure important details of the results (because $f(t)$ is large and rapidly increasing when $t$ is large), hence we have truncated the plot at $T=7$.

\subsection{The initialization problem}
\label{Section_ignition}
One drawback of PR is if initially the oscillator contains no initial energy ($x(0)=\dot{x}(0)=0$ in \eqref{eq_PR_canonical}) then parametric excitation has no effect. A remedy is to also add a nonparametric perturbation ($f(t)\neq 0$). For instance, if
\[
    A=\begin{bmatrix} 0 & 1 \\ -\omega^2 & 0 \end{bmatrix}, \qquad
    P(t)=\begin{bmatrix} 0 & 0 \\ -\omega^2 \cos(2 \omega t) & 0 \end{bmatrix},   \qquad
    f(t)=\begin{bmatrix} 0 \\ \delta \end{bmatrix} \quad (\delta\neq 0),
\]
an $\exp(\epsilon t/4)$ growth in the solution can be demonstrated by Theorem \ref{thm_temporalHomogenization}. This growth is due to the interaction between the small periodic and the nonparametric perturbations, because if either $P(t)$ or $f(t)$ is zero then the solution will not grow.

\section{Application 2: Energy harvest via parametric super-resonance and coupled RLC circuits}\label{sec:app2}

Consider the effect of  time-periodic oscillations in inductance or capacitance on the dynamic of
RLC circuits. For example, suppose that the capacitance fluctuates according to $\bar{C}(1-\eta \cos(2\omega t))$, where $\eta \ll 1$. It is known that the dynamic of such circuits is characterized by parametric resonance if $\omega \approx  \omega_n$, where $\omega_n$ is the intrinsic frequency of the oscillator. It can also be shown that, if $\omega = \omega_n$ and $2 R \bar{C} < \eta/\omega$ then the energy injected into the circuit by parametric resonance overcomes the dissipation induced by $R$, and the energy stored in circuit grows exponentially (see \cite{nonliearoscillations} for early experiments).

This phenomenon could, in principle, be used for energy harvesting. For instance, the earth-ionosphere behaves like a dielectric cavity with specific resonant frequencies. This leads to small oscillatory fluctuations in the ambient electromagnetic field at these frequencies \cite{price2007schumann}. Since these oscillations can result (through nonlinear effects) in oscillations of  circuit parameters, one natural question is the possibility of extracting the energy of these oscillations by
tuning the intrinsic frequency of the circuit  to hit parametric resonance (such questions can be traced back to Tesla's investigations on energy harvesting \cite{tesla2007experiments}).

The main limitation on the implementation of single parametrically-resonant circuit for harvesting energy is that the amplitude $\epsilon$ of induced parametrical fluctuations is usually too small to compensate the dissipative effect of the resistance ($2R\bar{C} \omega_n< \eta$ is needed for the compensation). We will use the temporal homogenization framework developed here to show that a large number of such circuits can, under the right coupling, overcome the dissipation.

\paragraph{Coupled RLC circuits.}
\begin{figure}[htb]
\centering
\footnotesize
\subfigure[The circuit.]{
\includegraphics[width=0.48\textwidth]{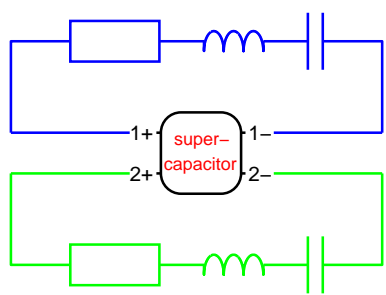}
\label{fig_circuit_a}
}
\subfigure[Schematic of the supercapacitor. Only conductive layers (electrodes) are shown; insulating dielectrics between adjacent layers are not drawn.]{
\includegraphics[width=0.48\textwidth]{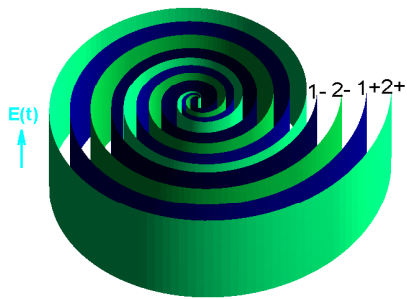}
\label{fig_circuit_b}
}
\caption{\footnotesize Coupled RLC circuits for energy harvest ($n=2$ for demonstration).}
\label{fig_circuit}
\end{figure}

Consider $n$ RLC circuits as illustrated in Figure \ref{fig_circuit_a},  coupled through the supercapacitor illustrated in Figure \ref{fig_circuit_b}. This supercapacitor is analogous to a wound film capacitor (e.g., \cite{bhattacharyya1982method}), where alternating conductive layers and dielectric layers are wound together, and it generates an electromotive force according to the sum of currents in all circuits, yet keeping these circuits insulated from each other.

Due to the electrostriction property of dielectrics (e.g., \cite{zhang1998giant}), the ambient electric field introduces a small periodic variation in the capacitance of this supercapacitor. This variation could be further enhanced, for instance, by attaching positive and negative charges respectively to two edges of electrodes via stiff nonconducting materials, which will stretch/compress the conducting plates according to the ambient electric field, and consequently change the capacitance (recall that parallel-plate conductor has a capacitance proportional to the plate area).

Denote by $I_i$ the current in the $i^{th}$ circuit. Assume the supercapacitor is symmetric with respect to permutations of electrodes (this is approximately true if sufficiently many turns are wound), such that the voltage difference $V_1$ across the public supercapacitor satisfies
\begin{equation}
    C(t) \frac{d V_1}{d t} = \sum_{i=1}^n I_i,
    \label{eq_publicCapacitor}
\end{equation}
where $C(t)=\bar{C}(1-\eta \cos(2\omega t))$ for some $\eta \ll 1$. Meanwhile, the voltage differences across the  capacitor, inductor, and resistor of sub-circuit $i$ respectively satisfy
\[
    C_i \frac{d V_{2,i}}{d t} = I_i, \qquad
    L_i \frac{d I_i}{d t} = V_{3,i}, \qquad
    R_i I_i = V_{4,i}.
\]
Kirchoff law of $V_1+V_{2,i}+V_{3,i}+V_{4,i}=0$ leads to the following dynamics:
\[
    \frac{1}{C(t)} \sum_{j=1}^n I_j + \frac{1}{C_i} I_i + R_i \frac{d I_i}{d t}+L_i \frac{d^2 I_i}{d t^2} = 0, \quad i=1,\ldots,n.
\]

%

\begin{Remark}
	The model described here is conceptual. For example, the choice of constant $\eta$ and $\omega$ is based on the assumption that the ambient electromagnetic fluctuations are sustained by an infinite energy reservoir. Also, when $n$ is large, it is an engineering challenge to pack all layers into a supercapacitor. 
\end{Remark}

\paragraph{Parametric super-resonance.}
For simplicity, consider identical circuits, i.e., $C_i=C$, $L_i=L$, $R_i=R$. Let $\epsilon=\eta/(L \bar{C})$, then $1/(L C(t))=1/(L \bar{C})+\epsilon \cos(2\omega t)+\mathcal{O}(\epsilon^2)$. Let $x=[I_1,\dot{I}_1,\cdots,I_n,\dot{I}_n]$, then
\[
    \dot{x}=Ax+\epsilon P(t)x+\mathcal{O}(\epsilon^2), \qquad \text{with}
\]
\[
    A=\begin{bmatrix} B & D & \cdots & D \\
      D & \ddots & \ddots & \vdots \\
      \vdots & \ddots & \ddots & D \\
      D & \cdots & D & B \end{bmatrix}
    \qquad\text{and}\qquad
    P=\begin{bmatrix} Q & Q & \cdots & Q \\
      Q & \ddots & \ddots & \vdots \\
      \vdots & \ddots & \ddots & Q \\
      Q & \cdots & Q & Q \end{bmatrix}, \qquad \text{where}
\]
\[
    B=\begin{bmatrix} 0 & 1 \\ -1/(LC)-1/(L\bar{C}) & -R/L \end{bmatrix}, \qquad
    D=\begin{bmatrix} 0 & 0 \\ -1/(L\bar{C}) & 0 \end{bmatrix}, \qquad
    Q=\begin{bmatrix} 0 & 0 \\ \cos(2\omega t) & 0 \end{bmatrix}.
\]
We will show that, provided $\omega=\sqrt{1/(LC)+n/(L\bar{C})-R^2/(4 L^2)}$, the solution grows exponentially if
\begin{equation}
\epsilon \frac{n}{\omega} > 2\frac{R}{L}, \quad\text{i.e.,}\quad \eta \frac{n}{\bar{C} \omega} > 2R,
\end{equation}
which is satisfied when $n$ (the number of coupled circuits) is large enough.

\begin{Lemma}
    \[
        \text{Let} \quad U=\begin{bmatrix} I & I & \cdots & I \\ I & -(n-1)I & I & \cdots & I \\ I & I & -(n-1)I & \ddots & \vdots \\ \vdots & \vdots & \ddots & \ddots & I \\ I & I & \cdots & I & -(n-1)I \end{bmatrix}, \text{ then }
        U^{-1}=\frac{1}{n} \begin{bmatrix} I & I & I & \cdots & I \\ I & -I & 0 & \cdots & 0 \\ I & 0 & -I & \ddots & \vdots \\ \vdots & \vdots & \ddots & \ddots & 0 \\ I & 0 & \cdots & 0 & -I \end{bmatrix},
    \]
    \[
        U^{-1}PU=\begin{bmatrix} nQ & & & & \\ & 0 & & & \\ & & 0 & & \\ & & & \ddots & \\ & & & & 0 \end{bmatrix}, \text{ and }
        U^{-1}AU=\begin{bmatrix} B+(n-1)D & & & & \\ & B-D & & & \\ & & B-D & & \\ & & & \ddots & \\ & & & & & B-D \end{bmatrix}.
    \]
    \label{lemma_coupledCircuitsLinearAlgebra}
\end{Lemma}
\begin{proof} Once the form of $U$ is obtained, the rest can be checked by simple algebra. \end{proof}

\begin{Lemma}
Let $\gamma=R/L$. If $\omega=\sqrt{1/(LC)+n/(L\bar{C})-\gamma^2/4}$, then
\begin{equation}
    \mathcal{G}\left[e^{-At}P(t)e^{At}\right]=n U \begin{bmatrix} \Delta & & & \\ & 0 & & \\ & & \ddots & \\ & & & 0 \end{bmatrix}U^{-1}, \text{ where }
    \Delta=\begin{bmatrix} \frac{\gamma}{8\omega^2} & \frac{1}{4\omega^2} \\ -\frac{\gamma^2-4\omega^2}{16\omega^2} & -\frac{\gamma}{8\omega^2} \end{bmatrix}.
    \label{eq_coupledCircuit_Delta}
\end{equation}
\label{lemma_coupledCircuitsB}
\end{Lemma}
\begin{proof}
    \[
        U^{-1} \mathcal{G} [\exp(-At)P(t)\exp(At)]U=\mathcal{G}\left[ \exp\left(-(U^{-1}AU)t\right) U^{-1}PU \exp\left((U^{-1}AU)t\right) \right].
    \]
    Using results in Lemma \ref{lemma_coupledCircuitsLinearAlgebra}, the above matrix has all 0 block-elements except for the first diagonal element, which is
    \[
        \Delta=\mathcal{G}[\exp(-(B+(n-1)D)t)nQ(t)\exp((B+(n-1)D)t)].
    \]
    Note $B+(n-1)D=\begin{bmatrix} 0 & 1 \\ -1/(LC)-n/(L\bar{C}) & -\gamma \end{bmatrix}$, whose eigenvalues are $-\frac{\gamma}{2} \pm \imath \omega$. Standard calculations lead to \eqref{eq_coupledCircuit_Delta}.
\end{proof}

\begin{Corollary}
    Given $x(0) \in \mathbb{R}^{2n} \setminus E$ for some $2n-1$ dimensional linear subspace $E\subset \mathbb{R}^{2n}$, $\|x(t)\|$ is unbounded if and only if $\epsilon \frac{n}{\omega}>2\gamma$.
\end{Corollary}
\begin{proof}
    Substitution of Lemma \ref{lemma_coupledCircuitsB} in Theorem \ref{thm_temporalHomogenization} leads to (as seen in the proof of Theorem \ref{thm_temporalHomogenization}, ignoring the $\mathcal{O}(\epsilon^2)$ term in the equation does not affect the leading term in the solution):
    \[
        x(t) \approx U \text{diag}\left[ e^{(B+(n-1)D)t}, e^{(B-D)t}, \cdots, e^{(B-D)t} \right] U^{-1} U \text{diag}\left[ e^{\epsilon n t \Delta}, I, \cdots, I \right] U^{-1} x(0).
    \]
    Since eigenvalues of $\Delta$ are $\pm \frac{1}{4\omega}$, real parts of eigenvalues of the above approximate solution operator are $\exp((-\frac{\gamma}{2} \pm \epsilon \frac{n}{4\omega})t)$ and $\exp(-\frac{\gamma}{2}t)$. The solution will be dominated by exponential growth if and only if $-\frac{\gamma}{2}+\epsilon \frac{n}{4\omega}>0$, unless $U^{-1}x(0)$ projects to zero in the direction of the $\Delta$ eigenvector associated with its $+\frac{1}{4\omega}$ eigenvalue.
\end{proof}
\begin{Remark}
	As initial conditions that do not lead to unbounded growth are of measure zero, in practice it is unlikely that they will hamper energy harvest. To entirely avoid this possibility, one can add to the system an `ignition', which is a short period forcing term (see section \ref{Section_ignition}).
\end{Remark}

\paragraph{On the constitutive variable capacitor equation.}
Using the equation $	\frac{d}{dt}(C(t)V_1(t))=\sum_{i=1}^n I_i $ instead of \eqref{eq_publicCapacitor} to model the shared supercapacitor leads to similar results, i.e. an exponential growth of the solution is achieved when $\omega=\sqrt{1/(LC_i)+n/(L\bar{C})-R^2/(4 L^2)}$ and $n$ is large enough. To sketch this calculation, note the system can be shown to be governed by
\[ \begin{cases}
	\dot{V}_1 &= (\sum I_i-\dot{C}(t)V_1)/C(t) \\
	\dot{V}_{2i} &= I_i / C_i \\
	\dot{I}_i &= -(V_1+V_{2i}+RI_i)/L
\end{cases}, \]
which can be written in canonical form (up to $\mathcal{O}(\eta)$) by letting $x=[V_1,V_{21},I_1,\cdots,V_{2n},I_n]$,
\[
    A=\begin{bmatrix} 0 & 0 & \frac{1}{\bar{C}} & 0 & \frac{1}{\bar{C}} & \cdots \\
    	0 & 0 & \frac{1}{C_i} & 0 & 0 & \cdots \\
    	-\frac{1}{L} & -\frac{1}{L} & -\frac{R}{L} & 0 & 0 & \cdots \\
    	0 & 0 & 0 & 0 & \frac{1}{C_i} & \\
    	-\frac{1}{L} & 0 & 0 & -\frac{1}{L} & -\frac{R}{L} & \\
    	\vdots & \vdots & \vdots & & & \ddots
    \end{bmatrix},\text{ and }
    P(t)=\begin{bmatrix} -2\omega\sin(2\omega t) & 0 & \frac{\cos(2\omega t)}{\bar{C}} & 0 & \frac{\cos(2\omega t)}{\bar{C}} & \cdots \\
    	0 & 0 & 0 & 0 & 0 & \cdots \\
    	0 & 0 & 0 & 0 & 0 & \cdots \\
    	0 & 0 & 0 & 0 & 0 & \cdots \\
    	0 & 0 & 0 & 0 & 0 & \cdots \\
    	& & \vdots & & & \ddots
    \end{bmatrix}.
\]
The following $U$ and $U^{-1}$ lead to block-diagonal $U^{-1}A U$ (with block sizes of $1,2,2,\cdots$):
\[
	U=\begin{bmatrix} \alpha & \beta & \mathbf{0} & \mathbf{0} & \cdots \\
		\gamma & I & -I & -I & \cdots \\
		\gamma & I &    &    & \iddots  \\
		\gamma & I &    & I  &    \\
		\vdots & \vdots & \iddots  &    &
	\end{bmatrix}, \quad \text{and} \quad
	U^{-1}\begin{bmatrix} x & y & y & y & \cdots \\
		\Delta & z & z & z & \cdots \\
		\mathbf{0} & -I/n & -I/n & -I/n & \iddots  \\
		\mathbf{0} & -I/n & -I/n & I-I/n & -I/n \\
		\vdots & \vdots & \iddots & -I/n & -I/n
	\end{bmatrix},
\]
\begin{align*}
	\text{where}\quad &\alpha=-\frac{E_{21}}{d_2}\lambda, \quad
	\beta=\begin{bmatrix} \frac{b_2}{E_{12}}n & 0 \end{bmatrix}, \quad
	\gamma=\begin{bmatrix} \lambda \\ 0 \end{bmatrix},
	\\
	&x=-\frac{d_2 E_{12}}{\zeta \lambda}, \quad
	y=\begin{bmatrix} \frac{b_2 d_2} {\zeta \lambda} & 0 \end{bmatrix}, \quad
	\Delta=\begin{bmatrix} \frac{d_2 E_{12}}{\zeta} \\ 0 \end{bmatrix}, \quad
	z=\begin{bmatrix} \frac{E_{12} E_{21}}{\zeta n} & 0 \\ 0 & \frac{1}{n} \end{bmatrix},
	\\
	&\text{with } \begin{bmatrix} b_1 & b_2 \end{bmatrix}=\begin{bmatrix} A_{12} & A_{13} \end{bmatrix}, \quad
	\begin{bmatrix} d_1 \\ d_2 \end{bmatrix}=\begin{bmatrix} A_{21} \\ A_{31} \end{bmatrix}, \quad
	\mathbf{E}:=\begin{bmatrix} E_{11} & E_{12} \\ E_{21} & E_{22} \end{bmatrix} = \begin{bmatrix} A_{22} & A_{23} \\ A_{32} & A_{33} \end{bmatrix},
	\\
	&\zeta=n b_2 d_2+E_{12}E_{21}, \text{ and } \lambda \text{ being an arbitrary nonzero scalar}.
\end{align*}
Once $U$ and $U^{-1}$ are explicitly identified, it can be computed that
\[
	U^{-1}AU = \begin{bmatrix} 0 & 0 & 0 & & & \\
		0 & 0 & \frac{1}{C_i} & & & \\
		0 & -\frac{\bar{C}+nC_i}{\bar{C}L} & -\frac{R}{L} & & & \\
		& & & \mathbf{E} & & \\
		& & & & \mathbf{E} & \\
		& & & & & \ddots
	\end{bmatrix},
\]
whose eigenvalue of $\omega$ resonates with the parametric perturbation, and that
\[
	U^{-1}P(t)U = \frac{1}{\bar{C}+n C_i} \begin{bmatrix}
		-2\omega\bar{C} \sin{2 \omega t} & \frac{1}{\lambda}2\omega C_i n \sin(2 \omega t) & -\frac{1}{\lambda} n \cos(2 \omega t) & 0 & 0 & \cdots \\
		2\omega \lambda \bar{C} \sin(2 \omega t) & -2\omega C_i n \sin(2 \omega t) & n \cos(2 \omega t) & 0 & 0 & \cdots \\
		0 & 0 & 0 & 0 & 0 & \cdots \\
		0 & 0 & 0 & 0 &   & \cdots \\
		0 & 0 & 0 &   & 0 & \cdots \\
		\vdots & \vdots & \vdots & & & \ddots
	\end{bmatrix}.
\]
Then parametric super-resonance can again be demonstrated by temporal homogenization.

\paragraph{A preliminary analysis of practical feasibility.}
The first mode of ambient electromagnetic fluctuations has its peak around $\sim$8Hz, with an electric field amplitude at the order of $10^{-3}V/m$ (c.f., static fair-weather electric field is about $150V/m$) \cite{price2007schumann}. This means $\omega$ is fixed and $\mathcal{O}(1)$, and it is reasonable to assume $\eta$ to be at the order of $10^{-5}$ or $10^{-6}$. We look for circuit parameters that satisfy
\begin{equation}
    \omega=\sqrt{\frac{1}{LC}+n\frac{1}{L\bar{C}}-\frac{R^2}{4 L^2}} \quad\text{and}\quad
    \eta \frac{n}{\bar{C} \omega} > 2R.
    \label{eq_harvest_constraints}
\end{equation}
Contemporary technologies can provide compact (super)capacitors \cite{Conway99} and inductors with values ranging from $10^{-12}$F to $10^4$F and $10^{-6}$H to $1$H. Writing $L=\eta^{l}[H], C=\eta^{a}[F],\bar{C}=\eta^{b}[F],n=\eta^{N}$, it is also reasonable to assume $R=\mathcal{O}(1)$ and constraints $-\sigma \leq l,a,b \leq \sigma$ for some positive parameter $\sigma$. Since $\eta \ll 1$, \eqref{eq_harvest_constraints} can be satisfied if leading order terms (in $1/\eta$) match, i.e.,
\[
    \min(-l-a,N-l-b)=2r-2l \quad \text{and} \quad 1+N-b<r.
\]
This linear programming problem is feasible when $\sigma\geq 1$. We choose to minimize $1+N-b-r$ in order to maximize the circuit gain, and then one solution is $l=\sigma$, $a=\sigma$, $b=\sigma/2$, $N=-\sigma/2$. When $\sigma=1$, this corresponds to parameters:
\[
    L=\mathcal{O}(\eta), \quad C=\mathcal{O}(\eta), \quad \bar{C}=\mathcal{O}(\sqrt{\eta}), \quad \text{ and } n=\mathcal{O}(1/\sqrt{\eta}).
\]
When $\eta\sim 10^{-6}$, this design requires the coupling of $\sim 10^{3}$ circuits to achieve energy gain. 

Although this preliminary analysis gives some indications on the workability of energy harvesting via super-resonance, it is by far incomplete, and a comprehensive feasibility analysis would require addressing possibly difficult engineering challenges such as (1) identifying workable physical configurations for packing a large number of layers into a supercapacitor and a large number of circuits around that supercapacitor  (2) keeping the financial cost of the system limited.
 These investigations are beyond the scope of this article.

\begin{wrapfigure}{r}{0.5\textwidth}
\vspace{-30pt}
\hspace{-30pt}
\footnotesize
\includegraphics[width=0.6\textwidth]{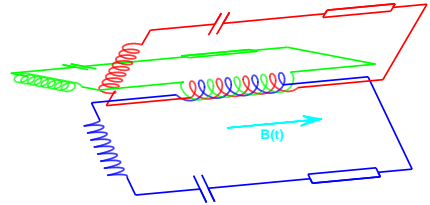}
\vspace{-40pt}
\caption{\footnotesize Alternative coupled RLC circuits for energy harvest. $n=3$ for demonstration; the shared ferromagnetic core of the inductors is not drawn.}
\vspace{-30pt}
\label{fig_circuit_alternative}
\end{wrapfigure}
\paragraph{Alternative design.}
We also note that similar scaling effects can also be achieved by coupling inductors. See Figure \ref{fig_circuit_alternative} for an illustrative design. Inductance can be coupled to ambient magnetic fluctuations if, for instance, the inductors have a ferromagnetic core.

\paragraph{Acknowledgments}
This work was supported by NSF grant CMMI-092600, a generous gift from UTRC, and Courant Instructorship from New York University. We thank Wei Mao for knowledge in engineering aspects of Amplitude Modulation, G\'{e}rard Ben Arous, Emmanuel Frenod, Jonathan Goodman, Robert Kohn for stimulating mathematical discussions, and anonymous referees for helpful comments.

\footnotesize
\bibliographystyle{siam}
\bibliography{molei21}

\end{document}